\documentclass[a4paper,12pt]{article}
\usepackage{amsmath,amsfonts,amssymb,amsthm,amscd}
\usepackage[arrow, matrix, curve]{xy}
\usepackage[colorlinks=true,citecolor=blue]{hyperref}

\font\cmssl=cmss10 at 12 pt

\usepackage{slashed}
\usepackage[normalem]{ulem}

\newtheorem{thm}{Theorem}
\newtheorem{lem}[thm]{Lemma}
\newtheorem{prop}[thm]{Proposition}
\newtheorem{defn}[thm]{Definition}
\newtheorem{cor}[thm]{Corollary}
\newtheorem{rem}[thm]{Remark}
\newtheorem{notation}[thm]{Notation}

\newtheorem{ass}[thm]{Assumption}

\title {Darboux theorem  for generalized complex structures on transitive Courant algebroids}

\date{\today}

\author{Vicente Cort\'es and  Liana David}

\begin{document}

\maketitle

{\bf Abstract:} 
Under natural assumptions   we find local normal forms for 
generalized complex structures on transitive Courant algebroids, which extend 
Gualtieri's Darboux theorem for generalized complex structures on manifolds \cite{gualtieri-thesis,gualtieri-annals}.   
When the base of the Courant algebroid  is a point,  they  reduce to
Wang's description of 
invariant complex structures on compact semisimple Lie groups \cite{wang}.

\section{Introduction}

Generalized complex geometry is a well-established area of mathematics deeply rooted in differential, complex and symplectic geometry. Its starting point  \cite{H} was to replace the tangent bundle of a manifold  $M$ with the generalized tangent bundle 
  $\mathbb{T}M = TM\oplus T^{*}M$ 
and   to consider   complex and symplectic structures as particular classes  of a more  general 
type of structure,   defined on the generalized tangent bundle  $\mathbb{T}M$
and called \emph{generalized complex structure} on $M$.
Many notions from classical differential geometry were   defined and studied, in a first instance on generalized tangent bundles and then on arbitrary Courant algebroids
(see e.g. \cite{gualtieri-thesis,gualtieri-annals, mario}).

A key result of the theory is Gualtieri's   generalized Darboux theorem \cite{gualtieri-thesis,gualtieri-annals}, which states that any generalized complex structure on a manifold is
locally (modulo diffeomorphisms  and $B$-field transformations),  around regular points,  the direct sum of a complex and a symplectic structure. 
Our aim in this paper is to develop a similar local description of   generalized complex structures on transitive Courant algebroids.

\begin{defn} A  generalized almost complex structure on a Courant algebroid $E$ is a skew-symmetric field of endomorphisms $\mathcal J \in \Gamma ( \mathrm{End}\, E)$ 
which satisfies $\mathcal J^{2} = -\mathrm{Id}.$ 
The generalized  almost complex structure $\mathcal J$ is {\cmssl  integrable} 
or is a {\cmssl generalized complex structure}
 if the  space of sections of its $(1,0)$-bundle $L$  is closed under the Dorfman bracket of $E$.\
\end{defn}

In view of Gualtieri's generalized Darboux theorem, one may ask  if  a generalized complex structure  $\mathcal J$ on a transitive Courant algebroid  $E$ is locally equivalent 
(under suitable regularity assumptions) 
to
 the direct sum of
a complex structure, a symplectic structure and a constant  complex structure on the  fiber type of the quadratic Lie algebra bundle associated to the Courant algebroid.
We give  a positive answer to this question, provided that  a certain Lagrangian subbundle $\mathcal D$ canonically associated to
$\mathcal J$  
has real index $p=0$.  Without the assumption $p=0$, the local form of $\mathcal J$ is in general more involved and we describe it in terms of the $(1,0)$-bundle 
$L$ of $\mathcal J$   (see our main Theorem \ref{main-thm}). 
When  the base of $E$  is a point,  our result  reduces to
Wang's description of (skew-symmetric)  left-invariant complex structures on compact semisimple Lie groups \cite{wang}.\\

{\bf Structure of the paper.} 
In Section \ref{prelim-section} we  fix our general assumptions and we revise various facts we need on the theory of Courant algebroids.  Along the paper we follow the notation and conventions from
\cite{cortes-david-JSG}, where more details on this material can be found.

In Section  \ref{standard-untwisted-section} we  prove that any 
transitive Courant algebroid is locally isomorphic to an untwisted Courant algebroid. In view of 
the classification of regular Courant algebroids  \cite{chen}, 
an untwisted Courant algebroid  over a manifold $M$ is essentially (cf.\ Definition~\ref{untwisted:def} and Notation~\ref{untwisted:notation}) a standard Courant algebroid   for which the  quadratic Lie algebra bundle $\mathcal G$ is the trivial bundle
$M\times \mathfrak{g}$,  the connection 
$\nabla$  on $\mathcal G$   is the canonical connection and the twisting forms  $H\in \Omega^{3}(M)$ and $R\in \Omega^{2} (M, \mathcal G)$ are  trivial
($H=0$, $R=0$).\

In Sections \ref{dirac-section} and  \ref{dirac-complex-section} we describe Dirac structures on direct sums  $V\oplus V^{*} \oplus \mathfrak{g}$  and their complexification,  
where  $(\mathfrak{g}, \langle \cdot , \cdot \rangle_{\mathfrak{g}} )$ is a vector space with scalar product of neutral signature\footnote{(to be specialised later to a Lie algebra with invariant neutral scalar product)},
$V$ is a real vector space  and $V\oplus V^{*}$ is endowed with its natural metric of neutral signature. 
As an application, 
we obtain a description of the $(1,0)$-space of a skew-symmetric complex structure on $V\oplus V^{*} \oplus \mathfrak{g}$, in terms of
a subspace $W$ of $V^{\mathbb{C}}$,  
a maximal isotropic subspace 
$\mathcal D$ of $\mathfrak{g}^{\mathbb{C}}$,  a complex-linear map $\sigma : W \rightarrow \mathfrak{g}^{\mathbb{C}}$  and a $2$-form $\epsilon \in \Lambda^{2}W^{*}.$  The data $(W,\mathcal D , \sigma ,  \epsilon)$ are subject to certain nondegeneracy conditions described in Corollary \ref{index-zero} and Proposition~\ref{basis}.\

In Section \ref{bundle-section}  we consider a generalized  almost complex  structure   $\mathcal J$ defined on a  standard 
Courant algebroid $E = TM\oplus T^{*}M\oplus \mathcal G$ and we describe the $(1,0)$-bundle of  $\mathcal J$ 
in terms of a   subbundle $W\subset ( TM) ^{\mathbb{C}}$, 
a maximal isotropic subbundle $\mathcal D$ of $\mathcal G^{\mathbb{C}}$, 
a linear map $\sigma : W \rightarrow \mathcal G^{\mathbb{C}}$
and a $2$-form $\epsilon \in \Gamma ( \Lambda^{2} W^{*})$ (see Assumption 
\ref{Wsmooth:ass} and Proposition \ref{prop-smooth}).
In terms of these data, $L$ is  given by   
\begin{align}
\nonumber& L =\{ X +\xi + \sigma (X)+ r\mid X\in W,\ \xi\in  ((TM)^{\mathbb{C}})^{*},\ r\in \mathcal D, \\
\nonumber&  \xi (Y) = 2 \epsilon (X, Y) - \langle \sigma (Y) ,  \sigma (X) + 2r\rangle_{\mathfrak{g}},\ \forall Y\in W\} .
\end{align}
In   Proposition~\ref{conditii-integr}  we find  the conditions on 
$(W, \mathcal D , \sigma , \epsilon)$ for  the integrability of $\mathcal J .$   In particular, we obtain that $\mathcal D$ is a Lagrangian subbundle
of $\mathcal G^{\mathbb{C}}.$ 

In Section \ref{classif-local} we  apply the material from the previous sections to prove the main result of this paper (Theorem \ref{main-thm} below). 
We  consider only  transitive Courant algebroids $E$ for which the   fiber type   $\mathfrak{g}$ of the quadratic Lie algebra bundle is a compact semisimple Lie algebra. 
We also assume that the Lagrangian subbundles $\mathcal D$  of the generalized  complex  structures  are mapped to trivial bundles under local isomorphisms 
between $E$ and untwisted Courant algebroids  and that the fibers of
$\mathcal D$ are regular subalgebras of $\mathfrak{g}^{\mathbb{C}}.$
 We call a  generalized complex structure with these properties regular 
 (see Definition~\ref{norm}). For motivations of these assumptions,  see  Remark  \ref{wang-rem}.
 We recall that a subalgebra of a semisimple Lie algebra is called regular if it is normalised by a Cartan subalgebra \cite{Dyn}.
We prove: 

\begin{thm}\label{main-thm} Any regular generalized complex structure 
is locally equivalent, in a neighborhood of a regular point,  to one defined  on an untwisted Courant algebroid $E= TM\oplus T^{*}M\oplus (M\times \mathfrak{g})$,
where $M = U_{1}\times U_{2}\times V$, $U_{1}\subset \mathbb{R}^{p}$ with coordinates $(x^{i})$, $U_{2}\subset \mathbb{R}^{2q}$ with coordinates
$(y^{i})$  and $V\subset \mathbb{C}^{k}$  with coordinates $(z^{i})$ are open neighborhoods 
of the origins, with $(1,0)$-bundle $L = L(W, \mathcal D ,\sigma , \epsilon )$ as follows:  
\begin{enumerate}
\item \begin{equation}
W = \mathrm{span}_{\mathbb{C}}\left\{ \frac{\partial}{\partial x^{i}}, 
(1\leq i\leq p),\ \frac{\partial}{\partial y^{i}} (1\leq i\leq 2q),\  \frac{\partial}{\partial \bar{z}^{j}},\ (1\leq j\leq k)  \right\}  ;
\end{equation}
\item $\mathcal{D}_{x}\subset\mathfrak{g}^{\mathbb{C}}$ is independent of $x\in M$ and  corresponds to  the 
$(1,0)$-space of an invariant complex structure of $G/T$, where $G$ is a  connected Lie group with Lie algebra $\mathfrak{g}$ and $T$ is a torus 
of dimension $p$;\
\item $\sigma = \sum_{i=1}^{p}dx^{i} \otimes (\sqrt{-1}  {v}^{\prime}_{i})$, where $(v^{\prime}_{i})$  are constant and form a basis of an  isotropic subspace 
$\mathfrak{h}_{\mathfrak{g}}^{\mathcal C}
\subset  \mathfrak{h}_{\mathfrak{g}}$ such that  $\langle \cdot , \cdot \rangle_{\mathfrak{g}}\vert_{\mathfrak{h}_{\mathfrak{g}}^{\mathcal D}\times \mathfrak{h}_{\mathfrak{g}}^{\mathcal C}}$
is non-degenerate. Here  $\mathfrak{h}_{\mathfrak{g}}$ is a Cartan subalgebra of $\mathfrak{g}$ which normalises $\mathcal D := \mathcal D_{x}$ and 
$\mathfrak{h}_{\mathfrak{g}}^{\mathcal D} := \mathfrak{h}_{\mathfrak{g}}\cap \mathcal D$;\
\item \begin{equation}
\epsilon =  \sqrt{-1}  ( \sum_{i=1}^{q} dy^{i}\wedge dy^{i+q}). 
\end{equation}
\end{enumerate}
\end{thm}

In Section \ref{dep-param-sect} we prove:

\begin{cor} \label{dep-param-cor}The parameters $(q,k)$ together  with the orbit  $[\mathcal D ]$ of $\mathcal D$ under the action of 
$\mathrm{Aut} (\mathfrak{g})$ represent  a complete set of invariants of  the equivalence classes of regular generalized complex structures. 
The parameter $p=\dim M -2(q+k)$ is redundant and coincides with $\mathrm{dim}_{\mathbb{C}} \, (\mathcal D \cap \bar{\mathcal D}).$  
\end{cor}

As observed in Section \ref{dep-param-sect}, the orbit space of $\mathrm{Aut} (\mathfrak{g})$ on the Grassmannian of Lagrangian subalgebras $\mathcal D$ of $\mathfrak g^\mathbb{C}$ reduces to the 
orbit space of the symmetry group of the Dynkin diagram on the Grassmannian of maximally isotropic subspaces of $\mathfrak h_{\mathfrak g}^\mathbb{C}$.

In  the appendix  we state  auxiliary results   which are used in the main text.\

In our recent work \cite{components} 
we approach generalized complex structures  from the view-point  of endomorphisms (rather than $(1,0)$-bundles)  and we address 
global questions.  
Some of the results of \cite{components} may be seen as an application of Sections  \ref{dirac-section}  and \ref{bundle-section}.

\bigskip

{\bf Acknowledgements.} Research of VC is funded by the Deutsche\linebreak Forschungsgemeinschaft 
(DFG, German Research Foundation) under Germany's Excellence Strategy, EXC 2121 ``Quantum Universe,'' 390833306 and under -- SFB-Gesch\"aftszeichen 1624 -- Projektnummer 506632645.

\section{Preliminary material}\label{prelim-section}

Along the paper we make the following assumptions.

\begin{ass}\label{essential}{\rm 
We only consider transitive  Courant algebroids  with scalar product of neutral signature
and for which  the adjoint   representation 
$$
\mathrm{ad} : \mathcal G \rightarrow \mathrm{Der}_{\mathrm{sk}}\,  (\mathcal G ),\ X\mapsto \mathrm{ad}_{X} (Y) := [ X, Y]_{\mathcal G}
$$
of the quadratic Lie algebra bundle  is an isomorphism onto the bundle 
of skew-symmetric derivations. The latter condition holds, for instance, if the fibers of the  quadratic Lie algebra 
 bundle are semisimple.  We  denote by  $\mathrm{Aut} (\mathcal G )$ the  automorphism group of  $\mathcal G$ which preserves both 
its Lie bracket $[\cdot , \cdot ]_{\mathcal G}$ and scalar product $\langle \cdot , \cdot \rangle_{\mathcal G}.$ We use the same symbol for a tensor or tensor  field and its complex-linear extension.}   
\end{ass}

Let $E= TM\oplus T^{*}M\oplus \mathcal G$ be a standard Courant algebroid, with quadratic Lie algebra bundle
$(\mathcal G , [\cdot , \cdot ]_{\mathcal G}, \langle \cdot , \cdot \rangle_{\mathcal G})$  and defining data
$(\nabla , R, H).$  
Recall that $\nabla$ is a connection on $\mathcal G$ which preserves the Lie bracket  $[\cdot , \cdot ]_{\mathcal G}$ and scalar product 
$\langle \cdot , \cdot \rangle_{\mathcal G}$, $R\in \Omega^{2} (M, \mathcal G)$, $H\in \Omega^{3}(M)$  
are such that 
\begin{equation}\label{def-data}
R^{\nabla}(X, Y) r = [R(X, Y), r]_{\mathcal G},\ d^{\nabla}R =0,\ d H = \langle R\wedge R\rangle_{\mathcal G},
\end{equation}
for any $X, Y\in {\mathfrak  X} (M)$, $r\in \Gamma (\mathcal G)$, 
where $R^{\nabla}$ denotes the  curvature of $\nabla$.
The Dorfman bracket of $E$ is given by 
\begin{align}
\nonumber& [X, Y] = \mathcal L_{X}Y + i_{Y} i_{X} H + R(X, Y)\\
\nonumber& [ X, r] =  - 2\langle i_{X}R, r\rangle_{\mathcal G} + \nabla_{X} r \\
\nonumber& [r_{1}, r_{2} ] = 2 \langle \nabla r_{1}, r_{2}\rangle_{\mathcal G} +  [r_{1}, r_{2} ]_{\mathcal G} \\
\label{dorfman1}& [X, \eta ] = \mathcal L_{X} \eta ,\ [\eta_{1}, \eta_{2} ] = [ r, \eta ] =0,
\end{align}
for any $X , Y\in {\mathfrak X}(M)$,   $\eta , \eta_{1}, \eta_{2} \in \Omega^{1}(M)$,  $r, r_{1} r_{2} \in \Gamma (\mathcal G)$ and 
\begin{equation}\label{dorfman2}
[u, v] + [v, u] = 2 d \langle u, v\rangle ,\ \forall u, v\in \Gamma (E),
\end{equation} 
where 
$$
\langle  X+\xi + r_{1}, Y +\eta + {r}_{2}  \rangle = \frac{1}{2} (\xi (Y) +\eta (X)) +\langle r_{1}, r_{2} \rangle_{\mathcal G}
$$ 
is the scalar product of $E$. 
Any
transitive Courant algebroid is isomorphic to a standard Courant algebroid (see \cite{chen}).

A (fiber-preserving)  Courant algebroid isomorphism $I : E_{1} \rightarrow E_{2}$ between two standard Courant algebroids 
over a manifold $M$, with 
quadratic Lie algebra bundles $(\mathcal G_{i}, [\cdot  , \cdot ]_{\mathcal G_{i}}, \langle \cdot , \cdot \rangle_{\mathcal G_{i}} )$ and
defining  data  $(\nabla^{i}, R_{i}, H_{i})$,  is determined by a system $(K, \Phi  , \beta)$ 
 (called components of $I$), 
where
$K\in \mathrm{Isom} (\mathcal G_{1}, \mathcal G_{2})$ is an isomorphism of quadratic Lie algebra bundles, 
$\Phi \in \Omega^{1}(M, \mathcal G_{2})$,  $\beta \in \Omega^{2}(M)$, 
as follows:
\begin{align}
\nonumber& I(X) = X+  i_{X}\beta - \Phi^{*}\Phi  (X) +\Phi (X) \\
\nonumber& I(\eta ) =\eta\\
\label{def-iso}& I(r)  = - 2 \Phi^{*} K(r) + K(r),
\end{align}
for any $X\in {\mathfrak X}(M)$,
$\eta \in \Omega^{1}(M)$, $r\in \Gamma (\mathcal G_{1})$. 
Above $\Phi^{*} : \mathcal G_{2} \rightarrow T^{*}M$ is defined by 
$$
(\Phi^{*} r) (X):= \langle  r, \Phi (X)\rangle_{\mathcal G_{2}},\ \forall  r\in \mathcal G_{2},\ X\in  {\mathfrak X}(M).
$$
The  components  $(K, \Phi , \beta)$ satisfy 
\begin{align}
\nonumber& \mathrm{ad}_{\Phi (X)} = K\circ \nabla^{1}_{X}\circ K^{-1} -\nabla^{2}_{X},\\
\label{def-cond}& H_{2} = H_{1} - d\beta -  \langle (KR_{1} + R_{2} )\wedge \Phi\rangle_{\mathcal G_{2}} + c^{\Phi}
\end{align}
where  $X\in {\mathfrak X}(M)$ and  $c^{\Phi}(X, Y, Z):= \langle \Phi (X), [ \Phi (Y), \Phi (Z) ]_{\mathcal G_{2}} \rangle_{\mathcal G_{2}}$.
The right hand side of  the first relation in (\ref{def-cond}) is a skew-symmetric derivation and from Assumption \ref{essential} i) 
 $\Phi$ is uniquely determined by $K$ and $\nabla^{i}$. From  \cite[Lemma 5]{cortes-david-JSG} and Assumption \ref{essential} i) again,  
 the first relation in (\ref{def-cond})
 implies 
\begin{equation}\label{def-cond-1}
R_{2}(X, Y) = K R_{1}(X, Y) -(d^{\nabla^{2}} \Phi )(X, Y) -  [\Phi (X), \Phi (Y) ]_{\mathcal G_{2}}.
\end{equation}

If  $f : M \rightarrow N$ is a smooth map and $E$ is a standard Courant algebroid over $N$, with quadratic Lie algebra bundle
$(\mathcal G , [\cdot , \cdot ]_{\mathcal G}, \langle \cdot , \cdot \rangle_{\mathcal G} )$ and defining data $(\nabla , R, H)$, then 
the pullback quadratic Lie algebra bundle $(f^{*}\mathcal G , f^{*}  [\cdot , \cdot ]_{\mathcal G},$  $f^{*}\langle \cdot , \cdot \rangle_{\mathcal G} )$
together with  $(f^{*} \nabla , f^{*}R, f^{*}H)$ define the {\cmssl pullback Courant algebroid} 
$f^{!}E$  to $M$
(see  \cite{poisson} and  \cite[Lemma 27]{cortes-david-JSG}). 
If $f$ is a diffeomorphism, there is a natural map $f^{!} : E \rightarrow f^{!} E$  
(called a {\cmssl pullback}) which
maps $X_{f(x)} +\xi_{f(x)} + r_{f(x)}$ from $E_{f(x)} = T_{f(x)}N \oplus T^{*}_{f(x)}N\oplus \mathcal G_{f(x)}$ 
to $(d_xf)^{-1} (X_{f(x)}) + \xi_{f(x)}\circ (d_{x} f)   + r_{f(x)}$
from  $(f^{!}E)_{x} = T_{x}M \oplus T^{*}_{x}M\oplus  (f^{*}\mathcal G )_{x}
$, for any $x\in M.$ 
An  {\cmssl equivalence} 
$I: E_{1} \rightarrow E_{2}$ 
between two Courant algebroids  is a composition of  isomorphisms and
pullbacks.  Two Courant algebroids are {\cmssl equivalent} (respectively,  {\cmssl isomorphic}) if there is an equivalence (respectively, isomorphism) between them.

\section{Untwisted Courant algebroids}\label{standard-untwisted-section}

\begin{defn}\label{untwisted:def}
Let $ E = TM\oplus T^{*}M \oplus \mathcal G$ be a standard Courant algebroid with 
quadratic Lie algebra bundle $(\mathcal G , [\cdot , \cdot ]_{\mathcal G}, \langle \cdot , \cdot \rangle_{\mathcal G})$ and 
defining data $(\nabla , R, H).$  Then  $E$ is called  {\cmssl untwisted}  if 
$R =0$ and $H =0$ (in particular, $\nabla$ is flat). 
\end{defn}

 \begin{prop} \label{transCA_loc_triv:prop} Any transitive  Courant algebroid is locally isomorphic to an untwisted Courant algebroid
 with quadratic Lie algebra bundle the trivial bundle 
 $\mathcal G=M\times \mathfrak g$   with fiber  a quadratic Lie algebra $(\mathfrak g, [\cdot , \cdot ]_{\mathfrak g},\langle
\cdot , \cdot \rangle_{\mathfrak g})$ and $\nabla$  the  canonical (flat)  connection on $ \mathcal G$.
\end{prop}

\begin{proof} 
It is sufficient to prove the statement for standard Courant algebroids. 
Let $E = TM\oplus T^{*}M\oplus \mathcal G$ be a standard Courant algebroid with 
quadratic Lie algebra bundle $(\mathcal G , [\cdot , \cdot ]_{\mathcal G}, \langle \cdot , \cdot \rangle_{\mathcal G})$ and 
defining data $(\nabla , R, H).$ Let $U\subset M$ open and sufficiently small such that
$\mathcal G\vert_{U}$ admits a trivialization $(r_{i})$ in which $\langle \cdot , \cdot \rangle_{\mathcal G}$ 
and $[\cdot , \cdot ]_{\mathcal G}$ are constant
(see \cite[Proposition 90]{cortes-david-JSG}).
In this trivialization, 
we write $\nabla = d +\Omega$, where 
$\Omega \in \Omega^{1}(U, \mathcal G\vert_{U}).$ 
Since $\nabla$ preserves   $ [\cdot , \cdot ]_{\mathcal G}$ and  $\langle \cdot , \cdot \rangle_{\mathcal G}$, 
$\Omega_{X}\in \mathrm{Der}_{\mathrm{sk}} (\mathcal G\vert_{U})$,  for any $X\in \mathfrak{X}(U)$. From 
Assumption \ref{essential} i) we obtain that  $\Omega_{X} = \mathrm{ad}_{\alpha (X)}$, where $\alpha \in \Omega^{1} (U, \mathcal G\vert_{U} ).$  
Let $I\in \mathrm{End}\, (TU \oplus T^{*}U \oplus \mathcal G\vert_{U} )$ be defined by 
(\ref{def-iso}) with  
$(K:= \mathrm{Id}, \Phi := \alpha  , \beta :=0).$
Then $I$ is a Courant algebroid isomorphism from $E$ to another standard Courant algebroid $E_{1}$ with the same quadratic Lie algebra bundle.
Let  $(\nabla^{1},  R_{1}, H_{1})$ be the defining data for $E_{1}.$ Then, for every $X\in {\mathfrak X}(U)$ and 
$r\in \Gamma (\mathcal G\vert_{U})$, 
\begin{align}
\nabla^{1}_{X} r = \nabla_{X} r +[r, \alpha (X) ]_{\mathcal G} = X(r) + [\alpha (X), r]_{\mathcal G} +  [r, \alpha (X)]_{\mathcal G}= X(r),
\end{align}
i.e.\ $\nabla^{1}$ is the canonical connection defined by the trivialization $(r_{i}).$
Therefore,  without loss of generality we can assume from the very beginning that  $\mathcal G\vert_{U}$
admits a $\nabla$-parallel  trivialization in which $\langle \cdot , \cdot \rangle_{\mathcal G}$ 
and $[\cdot , \cdot ]_{\mathcal G}$ are constant.
Since $\nabla$ is flat,  from Assumption \ref{essential} i) 
 $R=0$ 
 and $H$ is closed (see relations (\ref{def-data})).
 Let  $b\in \Omega^{2} (U)$ such that $H = d b$. Then 
$I\in \mathrm{End} (TU\oplus T^{*}U \oplus \mathcal G\vert_{U})$ defined by (\ref{def-iso}) with  
$(K:= \mathrm{Id}, \Phi := 0, \beta  := b)$,  is a Courant algebroid isomorphism from  $E$ to another standard Courant algebroid
$E_{2}$, with the same quadratic Lie algebra bundle. Let $(\nabla^{2}, R_{2}, H_{2})$ be the defining data for $E_{2}.$ Then 
$\nabla^{2} = \nabla$, $R_{2} = R =0$  and $H_{2} =0$, see relations (\ref{def-cond}) and (\ref{def-cond-1}).
In particular, $E_{2}$ is untwisted. 
The claim follows from the next lemma with  isomorphism  
$K: \mathcal G\vert_{U} \rightarrow U\times \mathfrak{g}$ determined by the trivialization 
$(r_{i}).$ 
\end{proof}

\begin{lem}\label{iso-untwisted}
Let  $E_{i }= TM\oplus T^{*}M\oplus  \mathcal G_{i}$  ($i=1,2$) be two untwisted  Courant algebroids with quadratic Lie algebra
bundles $(\mathcal G_{i}, \langle \cdot , \cdot \rangle_{\mathcal G_{i}}, [\cdot , \cdot ]_{\mathcal G_{i}})$ and (flat) connections
$\nabla^{i}$ on $\mathcal G_{i}.$ 
Any  isomorphism $K\in \mathrm{Isom}\, (\mathcal G_{1}, \mathcal G_{2})$ of quadratic Lie algebra bundles 
can be extended locally to a Courant algebroid isomorphism  
$I:E_{1} \rightarrow E_{2}$,  unique up to
an exact $2$-form.
\end{lem}

\begin{proof} 
We need to find $\Phi$ and $\beta$ such that relations (\ref{def-cond}) hold with  
$R_{i} =0$, $H_{i} =0$ and  the given $K$ and $\nabla^{i}$. 
As already mentioned,   $\Phi\in \Omega^{1}(M, \mathcal G_{2} )$ is uniquely determined by the 
first relation in (\ref{def-cond}). For the local existence of $\beta$,  we need to 
show that the $3$-form
\begin{equation}
c^{\Phi}(X, Y, Z)= \langle \Phi (X), [\Phi (Y), \Phi (Z) ]_{\mathcal G_{2}} \rangle_{\mathcal G_{2}},\ \forall X, Y, Z\in \mathfrak{X}(M) 
\end{equation}
is closed.   
From   relation (\ref{def-cond-1}) with $R_{i}=0$,
\begin{equation}
c^{\Phi}(X, Y, Z) = - \langle \Phi (X),  (d^{\nabla^{2}} \Phi )(Y, Z)   \rangle_{\mathcal G_{2}} = -\frac{1}{3} \langle  \Phi \wedge d^{\nabla^{2}} \Phi 
\rangle_{\mathcal G_{2}} (X, Y, Z).
\end{equation}
We deduce that  $c^{\Phi}$ is closed if and only if $\langle d^{\nabla^{2}}\Phi \wedge d^{\nabla^{2}}\Phi \rangle_{\mathcal G_{2}}=0$, 
or 
\begin{equation}
\Omega :=\langle  [\Phi, \Phi ]_{\mathcal G_{2}} \wedge [\Phi  , \Phi  ]_{\mathcal G_{2} } \rangle_{\mathcal G_{2}}=0,
\end{equation}
where we used that  $(d^{\nabla^{2}})^{2} =0$, since $\nabla^{2}$ is flat. 
Applying definitions,  
$$
i_{X} \Omega =\langle  [\Phi (X), \Phi ]\wedge [ \Phi , \Phi ]\rangle_{\mathcal G_{2}} +
\langle  [ \Phi , \Phi ]\wedge [ \Phi (X), \Phi ]\rangle_{\mathcal G_{2}}\\
= 2\langle  [\Phi (X) ,\Phi ] \wedge [ \Phi , \Phi ]\rangle_{\mathcal G_{2}}
$$
and
\begin{equation}
\frac{1}{2} i_{Y} i_{X} \Omega = \langle   [\Phi (X), \Phi (Y) ] , [ \Phi , \Phi ]\rangle_{\mathcal G_{2}} -
\langle  [ \Phi (X), \Phi ]\wedge
[ \Phi (Y), \Phi ] \rangle_{\mathcal G_{2}} ,
\end{equation}
where, in order to simplify notation, we wrote $[\cdot , \cdot ]$ instead of $[\cdot , \cdot ]_{\mathcal G_{2}}.$ 
By similar computations, we arrive at
\begin{align*}
\nonumber& \frac{1}{2} i_{V} i_{Z} i_{Y} i_{X} \Omega = \langle [ \Phi (X), \Phi (Y) ], [ \Phi (Z), \Phi (V ) ] \rangle_{\mathcal G_{2}}\\
\nonumber& 
- \langle [ \Phi (X), \Phi (Z) ] , [ \Phi (Y), \Phi (V ) ] \rangle_{\mathcal G_{2}} 
+ \langle [ \Phi (X), \Phi (V) ], [ \Phi (Y), \Phi (Z ) ]\rangle_{\mathcal G_{2}}
\end{align*}
which vanishes since $\langle \cdot , \cdot \rangle_{\mathcal G_{2}}$ is $\mathrm{ad}$-invariant and $[\cdot , \cdot ]_{\mathcal G_{2}}$ satisfies the Jacobi identity. 
\end{proof}

\begin{notation} \label{untwisted:notation} {\rm   From now on we assume that the
 quadratic Lie algebra bundle  of an untwisted Courant algebroid is the trivial bundle with  the canonical connection, as in
Proposition \ref{transCA_loc_triv:prop}. Since our main theorem is of local nature, there is no loss of generality in doing so
(in view of Proposition \ref{transCA_loc_triv:prop}).}
\end{notation}

\section{Dirac structures  of $V\oplus V^{*}\oplus \mathfrak{g}$}\label{dirac-section}

Let $V$ and  $\mathfrak{g}$  be  real vector spaces and $\langle\cdot ,\cdot \rangle_{\mathfrak{g}}$ a scalar product on $\mathfrak{g}$ of neutral signature.
We consider the direct sum $V\oplus V^{*} \oplus \mathfrak{g}$, with scalar product 
\begin{equation}
\langle X + \xi + r, Y+\eta + s  \rangle = \frac{1}{2} ( \xi (Y) +\eta (X)) + \langle r, s \rangle_{\mathfrak{g}},
\end{equation}
for any $X, Y\in V$, $\xi , \eta \in V^{*}$ and $r, s \in \mathfrak{g}.$  In this section  we   describe 
Dirac structures (i.e.\ maximal  isotropic
subspaces)  of $(V\oplus V^{*} \oplus \mathfrak{g} , \langle \cdot , \cdot \rangle )$.
For  $\mathfrak{g} = \{ 0\}$, we obtain the well-known description of Dirac structures of $V \oplus V^{*}$, see e.g.\ \cite{gualtieri-thesis,gualtieri-annals}.
We denote by $\pi_V$, $\pi_{V^*}$ and $\pi_{\mathfrak g}$ the projections onto the summands of $V\oplus V^{*} \oplus \mathfrak{g}$.\

Let $L\subset V\oplus V^{*} \oplus \mathfrak{g}$ be a Dirac structure  and $W := \pi_V(L) \subset V$.
For any $X\in W$, we define the affine subspace
\begin{equation} \label{DX:eq}\mathcal D_{X} :=\pi_{\mathfrak g} ((\pi_V|_L)^{-1}(X))\subset \mathfrak{g},\quad \pi_V|_L : L \to W.\end{equation}   
Let $\sigma_{0}: W \rightarrow \mathfrak{g}$ be a linear map such that 
$\sigma_{0} (X) \in {\mathcal D}_{X}$, for any $X\in W.$
Such a map can be obtained by choosing a complement $\mathcal K\subset L$ of $\ker \pi_V|_L$ and 
defining $\sigma_0 (X) := \pi_{\mathfrak g} (\pi_V|_{\mathcal K})^{-1} (X)$.

\begin{notation}{\rm
For $A, B\subset \mathfrak{g}$, we define $A+ B$ to be the set of vectors $a+ b$ where $a\in A$ and $b\in B.$}
\end{notation}

The next lemma is straightforward. 

\begin{lem}
The set  $\mathcal D_{0}$ is an isotropic subspace of $(\mathfrak{g}, \langle \cdot , \cdot \rangle_{\mathfrak{g}})$. For any $X\in W$,
 \begin{equation}\label{sigma-d0} 
 \mathcal D_{X} =  \sigma_{0} (X) + \mathcal D_{0}.
 \end{equation}
\end{lem}

\begin{prop}\label{vspace-max-iso} Any Dirac structure  of $V \oplus V^{*} \oplus \mathfrak{g}$ is of the form 
\begin{align}
\nonumber& L =\{ X +\xi + \sigma (X)+ r\mid X\in W,\ \xi\in V^{*},\ r\in \mathcal D, \\
\label{form-L}&  \xi (Y) = 2 \epsilon (X, Y) - \langle \sigma (Y) ,  \sigma (X) + 2r\rangle_{\mathfrak{g}},\ \forall Y\in W\} ,
\end{align}
where $W\subset V$ is a subspace,  $\mathcal D \subset \mathfrak{g}$ is a Dirac structure, $\sigma \in \mathrm{Hom}(W,\mathfrak g)$ and
$\epsilon \in \Lambda^{2} W^{*}$.
Conversely, any choice of such data  $(W,  \mathcal D ,\sigma ,  \epsilon )$ defines a Dirac structure  
\[ L = L(W,   \mathcal D ,\sigma ,  \epsilon )\subset V \oplus V^{*} \oplus \mathfrak{g}\] 
by \eqref{form-L}. The homomorphism $\sigma$ automatically satisfies the condition $\sigma (X)\in \mathcal D_X$, with respect to $L=L(W,   \mathcal D ,  \sigma  ,\epsilon )$, for all $X\in W$.
\end{prop}

\begin{proof}  Let $L\subset V \oplus V^{*} \oplus \mathfrak{g}$ be a Dirac structure. From the previous lemma, 
\begin{equation}
L \subset \{ X +\xi +  \sigma_{0} (X) + r \mid X\in W,\  \xi \in V^{*},\, r\in \mathcal D_{0}\} ,
\end{equation}
where  $\sigma_0 \in \mathrm{Hom}(W,\mathfrak g)$ is any linear map satisfying  
$\sigma_0(X) \in \mathcal D_X$ for all $X\in W$, 
$\mathcal D_{X}$ is determined by $L$ via  \eqref{DX:eq}
and $\mathcal D_{0}$ is isotropic. 
Consider  $X +\xi_{i} + \sigma_{0} (X) + r_{i}  \in L$ ($i=1,2$).  Their difference $(\xi_{1} -\xi_{2}) + (r_{1} - r_{2})$ belongs to $L$, and, since $L$ is isotropic,
any $Y +\eta + \sigma_{0}(Y) + r\in L$ is orthogonal to $(\xi_{1} -\xi_{2}) + (r_{1} - r_{2})$. We deduce that
\begin{equation}
\frac{1}{2} (\xi_{1} -\xi_{2}) (Y) + \langle  \sigma_{0}(Y) , r_{1} - r_{2}\rangle_{\mathfrak{g}} =0,
\end{equation}
where we used that $\langle r, r_{i} \rangle_{\mathfrak{g}} =0$ since $\mathcal D_{0}$ is isotropic. 
It follows that the bilinear form 
\begin{equation}\label{epsprime:eq}
\epsilon' : W\times W \rightarrow \mathbb{R},\ \epsilon'  (X, Y) := \frac{1}{2} \xi (Y) +\langle \sigma_{0} (Y), r\rangle_{\mathfrak{g}}
\end{equation}
is well-defined, i.e.\ independent of the choice of $\xi$ and  $r$ such that $X +\xi + \sigma_{0}(X) + r\in L.$ 
We proved that 
\begin{align}
\nonumber& L \subset  \{ X +\xi + \sigma_{0}(X) + r\mid X\in W,\ \xi\in V^{*},\ r\in \mathcal D_{0},\\
\label{form-L-1}&  \xi (Y) = 2  ( \epsilon' (X, Y) -\langle \sigma_{0}(Y) , r\rangle_{\mathfrak{g}} ),\ \forall Y\in W\} .
\end{align}
The dimension of  the vector space on the right-hand side of (\ref{form-L-1}) is equal to   
$$
\mathrm{dim}\, W + \mathrm{dim}\, \mathcal D_{0} +  (\mathrm{dim}\, V - \mathrm{dim}\, W)=  
\mathrm{dim}\, V + \mathrm{dim}\, \mathcal D_{0}\leq
\mathrm{dim}\, V  + \frac{1}{2} \mathrm{dim}\, \mathfrak{g}
$$
where in the last inequality we used that  $\mathcal D_{0}$ is isotropic (in particular,  of dimension at most
$\frac{1}{2}   \dim\, \mathfrak{g}$). It follows that $L$ is  isotropic   of dimension $\mathrm{dim}\, V  + \frac{1}{2} \dim\, \mathfrak{g}$ 
(i.e.\ maximal  isotropic) if and only if equality holds in (\ref{form-L-1}), the  vector space defined by the right-hand side of 
(\ref{form-L-1}) is isotropic and $\mathcal D_{0}$ is maximal isotropic. 
A straightforward computation shows that  the right-hand side of  (\ref{form-L-1}) is isotropic if and only if 
\begin{equation}
\epsilon' (X, Y) +\epsilon' (Y, X) +\langle \sigma_{0}(X), \sigma_{0}(Y) \rangle_{\mathfrak{g}} =0,\ \forall X, Y\in W.
\end{equation}
Let $\epsilon \in \Lambda^{2} W^{*}$ be the skew-part of  $\epsilon'.$ 
The claim that $L$ has the form \eqref{form-L}  follows by writing
$$
\epsilon' (X, Y ) = \epsilon (X, Y) -  \frac{1}{2} \langle \sigma_{0}(X), \sigma_{0} (Y)\rangle_{\mathfrak{g}}
$$
and replacing it into  (\ref{form-L-1}).  To finish the proof we observe that if $L$ is defined by \eqref{form-L}, then $\mathcal D_X = \sigma (X)+\mathcal D$ and, hence, 
$\sigma (X)\in \mathcal D_X$ and $\mathcal D_0 = \mathcal D$.
\end{proof}

Next we study   the map 
\[ (W,  \mathcal D ,\sigma   , \epsilon ) \mapsto L(W,  \mathcal D , \sigma , \epsilon ).\]
It is defined on the set $\mathcal S$  of tuples $(W,   \mathcal D , \sigma ,  \epsilon )$ consisting 
of a subspace $W\subset V$, a Dirac structure 
$\mathcal D \subset \mathfrak g$, 
$\sigma \in \mathrm{Hom}(W,\mathfrak g)$  and $\epsilon \in \Lambda^2 W^*$. 

\begin{prop} \label{bundle}
The set $\mathcal S$
is a union of smooth vector bundles $\mathcal{S}_d$, $d=0,\ldots, n=\dim V$, with the projection 
$(W,   \mathcal D ,\sigma ,   \epsilon ) \to (W,  \mathcal D)$.  
\end{prop}

\begin{proof} 
The vector bundles $\mathcal{S}_d$ are obtained by fixing the dimension $d=\dim W$. 
The base of the vector bundle $\mathcal{S}_d$ is the homogeneous real projective algebraic
variety $\mathrm{Gr}_d(V)\times \mathrm{Gr}_{\mathrm{Dirac}}(\mathfrak g)$, 
where $\mathrm{Gr}_d(V)$ denotes the Grassmannian of subspaces $W\subset V$ of dimension $d$ 
and $\mathrm{Gr}_{\mathrm{Dirac}}(\mathfrak g)$ the Grassmannian of maximal isotropic subspaces $\mathcal D \subset \mathfrak g$. 
Denoting by $\mathfrak W_d$ the pullback of the 
universal bundle of  $\mathrm{Gr}_d(V)$  via the canonical projection $\mathrm{Gr}_d(V)\times \mathrm{Gr}_{\mathrm{Dirac}}(\mathfrak g) \to \mathrm{Gr}_d(V)$ 
we see that 
\[ \mathcal S_d = (\mathfrak W_d^* \otimes \mathfrak{g}) \oplus \Lambda^2 \mathfrak W_d^*. \qedhere\]
\end{proof}
 Let us denote by $\mathfrak D$ the pull back of the universal bundle of $\mathrm{Gr}_{\mathrm{Dirac}}(\mathfrak g)$ 
 via the projection  $\mathrm{Gr}_d(V)\times \mathrm{Gr}_{\mathrm{Dirac}}(\mathfrak g) \to \mathrm{Gr}_{\mathrm{Dirac}}(\mathfrak g)$.

\begin{prop} \label{action_on_Sd:prop}The vector bundle $\mathrm{Hom}(\mathfrak W_d, \mathfrak D)$, considered as a bundle of vector groups, 
acts fiberwise and freely on the vector bundle  $\mathcal{S}_d$ by affine transformations $T_\gamma$, $\gamma \in 
\mathrm{Hom}(\mathfrak W_d, \mathfrak D)$. Denoting the base point of $\gamma$ by 
 $(W,\mathcal D)\in \mathrm{Gr}_d(V)\times \mathrm{Gr}_{\mathrm{Dirac}}(\mathfrak g)$ we have 
 $\gamma \in \mathrm{Hom}(W,\mathcal D)$ and  the action on 
 the fiber of $\mathcal S_d$ over $(W,\mathcal D)$ 
 is explicitly given by 
\begin{equation}\label{ADD}
T_\gamma (W,   \mathcal D, \sigma ,  \epsilon ) = (W,   \mathcal D ,\sigma +\gamma , \epsilon + \frac12\langle \sigma \wedge \gamma\rangle_\mathfrak g).
\end{equation}
The quotient under this action 
\[ \mathcal S_d/\mathrm{Hom}(\mathfrak W_d, \mathfrak D)\] 
is in natural bijection with the set $\mathfrak S_d$ of 
Dirac structures $L\subset V\oplus V^* \oplus \mathfrak g$ for which $\dim \pi_V(L)=d$.  
\end{prop}

\begin{proof}  We only prove the claim on  $\mathfrak S_d$. 
The other claims are straightforward.  We consider the surjective map 
\begin{equation}\label{ldmap:eq} \mathcal L_d : \mathcal S_d \to \mathfrak S_d,\quad (W,  \mathcal D, \sigma ,   \epsilon ) \mapsto L(W,  \mathcal D, \sigma ,   \epsilon )\end{equation} 
from Proposition \ref{vspace-max-iso}. Given $L=L(W,  \mathcal D ,  \sigma , \epsilon )\in \mathfrak S_d$, the 
data $W$ and $\mathcal D$ are completely determined by $L$. In fact, 
\begin{equation}\label{WD:eq}W=\pi_V(L)\quad\mbox{and}\quad \mathcal D = \pi_{\mathfrak{g}}(\ker \pi_V|_L).\end{equation}
Next we recall that  $\sigma  : W\rightarrow \mathfrak{g}$ is linear and, for any
$X\in W$, $\sigma (X) \in \mathcal D_{X} = \pi_{\mathfrak{g}} (\pi_{V}\vert_{L})^{-1}(X)$.
Any other map $\sigma_{1} : W \rightarrow \mathfrak{g}$ with this property is related to $\sigma$ by 
\begin{equation}\label{sigma1:eq}
\sigma_{1}  = \sigma +  \gamma ,
\end{equation} 
where $\gamma: W \rightarrow \mathcal D$ is linear.

Finally, expressing $L$ as $L=L(W, \mathcal D , \sigma_1, \epsilon_1)$ in terms of $\sigma_{1}$ instead of $\sigma$  we obtain that 
the corresponding bilinear form  $\epsilon_{1}'\in W^{*}\times W^{*}$ and 
its skew part $\epsilon_{1}  \in \Lambda^{2} W^{*}$ are related to the bilinear forms $\epsilon'$ and $\epsilon$ associated with $L$ and $\sigma$ 
by \begin{equation}\label{trafo:eq}
\epsilon_{1}' = \epsilon'  -  \langle \gamma , \sigma \rangle_{\mathfrak{g}},\ 
\epsilon_{1} = \epsilon  + \frac{1}{2} \langle \sigma \wedge \gamma \rangle_{\mathfrak{g}}.
\end{equation}
To see this it suffices to recall that $\epsilon'$ is defined by \eqref{epsprime:eq} where $\sigma_0=\sigma$ and $\epsilon_1'$ is defined similarly in terms of $\sigma_1$.  We also note that for any $X, Y\in W$, 
\begin{align}
\nonumber& \langle \gamma , \sigma\rangle_{\mathfrak{g}}   (X, Y) := \langle \gamma (X),\sigma (Y)\rangle_{\mathfrak{g}}\\
\nonumber& \langle \sigma \wedge \gamma\rangle_{\mathfrak{g}}(X, Y):= \langle \sigma  (X), \gamma (Y) \rangle_{\mathfrak{g}} 
- \langle \sigma (Y),  \gamma (X) \rangle_{\mathfrak{g}}.
\end{align}
Comparing \eqref{sigma1:eq} and \eqref{trafo:eq} with the definition of the affine action, we see that the map \eqref{ldmap:eq} 
induces a bijection 
\[  \mathcal S_d/\mathrm{Hom}(\mathfrak W_d, \mathfrak D) \to \mathfrak S_d.\qedhere\]
\end{proof}

\begin{cor} \label{complement_sigma:eq}The fiber $(\mathcal L_d)^{-1}(L)$ of the map  $\mathcal L_d : \mathcal S_d \to \mathfrak S_d$, $(W,\mathcal D ,   \sigma ,  \epsilon ) \mapsto L(W,  \mathcal D , \sigma ,   \epsilon )$ over a point $L=L(W,\mathcal D , \sigma ,  \epsilon)$ is an affine space modelled on 
$\mathrm{Hom}(W,\mathcal D)$. The choice of a complement $\mathcal C$ of $\mathcal D$ in $\mathfrak g$ uniquely
determines a point in that affine space by requiring the image of $\sigma$ to lie in $\mathcal C$.  
\end{cor}
\begin{proof} From Proposition \ref{action_on_Sd:prop} we see that $(\mathcal L_d)^{-1}(L)$ is precisely the orbit 
of $(W, \mathcal D , \sigma , \epsilon)$ under the affine action of $\mathrm{Hom}(W,\mathcal D)$. 
Since the action is free, the orbit is an affine space, i.e.\ a set with a simply transitive action of a vector group. 
Now let $\mathcal C$ be a complement of $\mathcal D$ in  $\mathfrak{g}$. 
Using the affine action we can modify $\sigma$ by addition of a uniquely determined element $\gamma \in \mathrm{Hom}(W,\mathcal D)$ 
to ensure that it has image in $\mathcal C.$ We conclude that the resulting tuple 
$(W,\mathcal D, \sigma +\gamma ,  \epsilon  +  \frac{1}{2} \langle \sigma \wedge \gamma \rangle_{\mathfrak{g}})\in (\mathcal L_d)^{-1}(L)$ is uniquely determined by $L$ and the choice of $\mathcal C$. 
\end{proof}

\section{Dirac structures  of $(V \oplus V^{*} \oplus \mathfrak{g})^{\mathbb{C}}$}\label{dirac-complex-section}

The arguments from the previous section 
remain true for Dirac structures  (i.e.\  maximal  isotropic  complex subspaces)  $L$ of $(V \oplus V^{*} \oplus \mathfrak{g})^{\mathbb{C}}$
with scalar product the complexification  of $\langle \cdot , \cdot  \rangle$.  More precisely,  
$L$  can be described in terms of 
a complex subspace $W\subset V^{\mathbb{C}}$, a Dirac structure 
$ \mathcal D$ of $ (\mathfrak{g}^{\mathbb{C}}, \langle \cdot , \cdot \rangle_{\mathfrak{g}})$, 
a complex linear map $\sigma : W  \rightarrow \mathfrak{g}^{\mathbb{C}}$, 
and a  complex $2$-form  
$\epsilon\in \Lambda^{2} W^{*}$,  as in Proposition~\ref{vspace-max-iso},  and  Propositions  \ref{bundle}, \ref{action_on_Sd:prop} and 
Corollary \ref{complement_sigma:eq}  remain true in this setting. Recall  that  $L$ has {\cmssl real index zero} if $L\cap \bar{L} = \{ 0 \}$, or 
$$(L\cap \bar{L} )_{\mathbb{R}}:= L\cap \bar{L} \cap (V\oplus V^{*}\oplus \mathfrak{g}) 
= \{ 0 \} .$$
In this case 
$L$ defines a skew-symmetric complex structure $J=J_L$ on $V \oplus V^{*} \oplus \mathfrak{g}$ with the eigenspaces 
$\mathrm{Eig}(J,i)=L$ and $\mathrm{Eig}(J,-i)=\bar L$. We denote by  $\mathrm{Re}\,  \sigma : W\rightarrow \mathfrak{g}$  and $\mathrm{Im}\,  \sigma : W\rightarrow \mathfrak{g}$  the maps 
\begin{align}
\nonumber& \mathrm{Re}\,   \sigma (X)  :=\frac{1}{2} ( \sigma (X)  +\overline{\sigma (X)}),\\
\nonumber&  \mathrm{Im}\,   \sigma (X)  :=\frac{1}{2\sqrt{-1} } ( \sigma (X)  - \overline{\sigma (X)}).
\end{align}
We define  $\bar{\sigma} : \bar{W} \rightarrow \mathfrak g^\mathbb{C}$  by
$\bar{\sigma} (X):= \overline { \sigma (\bar{X})}$ and 
$\bar{\epsilon}\in \Lambda^{2} \bar{W}^{*}$  
by $\bar{\epsilon} (X, Y) := \overline{\epsilon (\bar{X}, \bar{Y})}$.
Let
$\Delta : = (W \cap \bar{W})_{\mathbb{R}} = W \cap \bar{W}\cap V.$

\begin{lem}\label{index-zero} The Dirac structure $L = L(W, \mathcal D ,  \sigma , \epsilon )$ has real index zero  if and only if  the next conditions hold:\

A)  $W + \bar{W} = V^{\mathbb{C}}$;\

B)  $\mathcal D \cap \mathfrak{g} \cap  [ (\mathrm{Im}\,  \sigma )({\Delta} )]^{\perp} =\{ 0\}$;\

C) for any $X\in \Delta\setminus \{ 0\}$ and $r\in \mathcal D$ such that $ r -\bar{r} = \bar{\sigma}(X) -\sigma (X)$, there is
$Y\in \Delta$ such that
\begin{equation}
\mathrm{Im}\, (\epsilon (X, Y) ) 
-\frac{1}{2} \langle \mathrm{Re}\, \sigma  \wedge \mathrm{Im}\, \sigma \rangle_{\mathfrak{g}} (X, Y)
-\frac{1}{2} \langle r + \bar{r}, \mathrm{Im}\, \sigma (Y) \rangle_{\mathfrak{g}} \neq 0.
\end{equation}
\end{lem} 

\begin{proof} Projecting the equality  $L\oplus \bar{L} = ( V\oplus V^{*} \oplus \mathfrak{g})^{\mathbb{C}}$
onto $V^{\mathbb{C}}$ we obtain that A) holds when $L$ has real index zero.  We assume  from now on that 
A) holds. From the definition of $L$,  $X+ \xi +\sigma (X) + r\in ( L\cap \bar{L})_{\mathbb{R}}$ if and only if 
$X\in \Delta$,   $\xi \in V^{*}$ (i.e. $\xi\vert_{V}$ takes real values),  $r\in \mathcal D$, 
\begin{equation}\label{sigma0-xz}
\sigma (X) + r = \bar{\sigma}(X) + \bar{r},
\end{equation}
and 
\begin{align}
\nonumber& \xi (Y) = 2 {\epsilon}  (X, Y) -\langle {\sigma}(Y) ,  {\sigma}(X) + 2r\rangle_{\mathfrak{g}},\ \forall Y\in  {W} \\
\label{bar-z-prima} &\xi (Y) = 2 \bar{\epsilon}  (X, Y) -\langle \bar{\sigma}(Y) ,  \bar{\sigma}(X) + 2\bar{r}\rangle_{\mathfrak{g}},\ \forall Y\in  \bar{W} .
\end{align}
Using   $W + \bar{W} = V^{\mathbb{C}}$ we obtain from relations (\ref{bar-z-prima})  that  
$\xi$  is uniquely  determined by  $X$ and $r$. Conversely,  given  $X\in \Delta$ and $r\in \mathcal D$
such that  (\ref{sigma0-xz}) holds,  there is $\xi\in V^{*}$ 
such that  $X+ \xi +\sigma (X) + r\in ( L\cap \bar{L})_{\mathbb{R}}$ 
if and only if
\begin{equation}\label{im}
\mathrm{Im}\, ( \epsilon (X, Y)) -\frac{1}{2} \langle \mathrm{Re}\, \sigma  \wedge \mathrm{Im}\, \sigma  \rangle_{\mathfrak{g}}  (X, Y)
-\frac{1}{2} \langle r + \bar{r}, \mathrm{Im}\, \sigma (Y) \rangle_{\mathfrak{g}} =0,
\end{equation}
for all $Y\in \Delta$. 
We obtain that 
$(L\cap \bar{L})_{\mathbb{R}}$ contains no  non-zero vector 
of the form $\xi + r$ if and only if  B) holds and it  contains no vector $X +\xi + \sigma (X) + r$ with $X\neq 0$ if and only if C) holds. 
\end{proof}

Let 
$$
\mathcal B := \{ v_{1}, \ldots , v_{p}, w_{1}, \ldots , w_{q},   {v}^{\prime}_{1}, \ldots , {v}^{\prime}_{p}, \tau (w_{1}), \ldots , \tau (w_{q}) \}
$$
be a basis of $\mathfrak{g}^{\mathbb{C}}$  adapted to $\mathcal D$ 
(see Proposition \ref{max-iso-chevalley}) and $\mathcal C$ the complement of $\mathcal D$ 
 spanned by the vectors $\{ {v}^{\prime}_{i}$, $\tau (w_{j})\}$.
Choose   $\sigma : W \rightarrow\mathcal C$ and write it as
\begin{equation}
\sigma (X) = \sum_{i=1}^{p} \lambda_{i}(X) {v}^{\prime}_{i} +\sum_{j=1}^{q} \beta_{j} (X) \tau( w_{j}),\ \forall X\in W
\end{equation}
where $\lambda_{i}, \beta_{j} \in W^{*}.$   
 
 \begin{prop} \label{basis}    The Dirac structure $L$ has real index zero  if and only if 
 $W + \bar{W} = V^{\mathbb{C}}$ and one of the following situations holds:
 
 i) either $\Delta  =0$ and $\mathcal D  =  \mathrm{span}_{\mathbb{C}} \{ w_{1}, \ldots , w_{q}\}$ (that is,  $\mathcal D\cap \mathfrak{g}= \{ 0\}$);

 ii) or  $\Delta \neq 0$,  $\{  \mathrm{Im}\,  (\lambda_{i}\vert_{\Delta}) \mid 1\leq i\leq p \}$ are linearly independent and 
 \begin{equation}\label{omega:eq}
 \omega := \mathrm{Im}\, \epsilon  +\frac{1}{2\sqrt{-1}}\sum_{j} \epsilon _{j} \beta_{j} \wedge \bar{\beta}_{j}
 \in (\Lambda^2 (W\cap \bar W)^*)_\mathbb{R} = \Lambda^2 \Delta^*
 \end{equation}
 is non-degenerate on 
 $$
 \Delta_0 :=  \{ X\in \Delta \mid  \mathrm{Im}\, \lambda_{i} (X) =0 \} .
 $$
In particular, the dimension  of $\Delta_{0}$ is even. (Here and in the proof we write $\sqrt{-1}$ instead of $i$ to avoid a conflict of
notation with the index $i$.)
 \end{prop}

 \begin{proof}    i) When  $\Delta =\{ 0\}$, B) and C)  from   Corollary \ref{index-zero} reduce to
$ \mathcal D \cap \mathfrak{g} =\{0\}.$

 ii) Assume  that  $\Delta \neq \{ 0\}.$ Then,  for any $X\in \Delta$, 
 \begin{align}
\nonumber&  \mathrm{Re}\, \sigma (X) = \sum_{i} \mathrm{Re}\,  \lambda_{i} (X) {v}^{\prime}_{i} +\frac{1}{2} \sum_{j} ( \beta_{j} (X) 
 \tau (w_{j}) +\bar{\beta}_{j} (X) w_{j})\\
\label{re-im} & \mathrm{Im}\, \sigma (X) = \sum_{i} \mathrm{Im}\, \lambda_{i} (X) {v}^{\prime}_{i} +\frac{1}{2\sqrt{-1}} \sum_{j} ( \beta_{j}(X) \tau
 ( w_{j}) -\bar{\beta}_{j}(X) w_{j}). 
 \end{align}
 We deduce that
 \begin{equation}
\langle  \mathrm{Re}\, \sigma \wedge \mathrm{Im}\, \sigma \rangle_{\mathfrak{g}} =\frac{1}{2\sqrt{-1}} \sum_{j} \epsilon_{j} \bar{\beta}_{j}\wedge \beta_{j}.
 \end{equation}
 Using the second relation  in (\ref{re-im}), we see that $v= \sum_{i} a_{i} v_{i} \in
 \mathcal D \cap \mathfrak{g}$ is orthogonal to $\mathrm{Im}\, \sigma (X) $  if and only if
 $\sum_{i}  a_{i} \mathrm{Im}\, \lambda_{i} (X) =0$.  It follows that condition B) from
 Corollary \ref{index-zero} is equivalent to the linear independence of 
 $\{  (\mathrm{Im}\,  (\lambda_{i} \vert_{\Delta}) \mid 1\leq i\leq p\}$.  Let us consider now condition C) from Corollary~\ref{index-zero}. 
 It is straightforward to see that a vector 
 $$
 r= \sum_{i} z_{i}v_{i} + \sum_{j} z^{\prime}_{j} w_{j} \in \mathcal D 
 $$
 (where $z_{i}, z_{j}^{\prime} \in \mathbb{C}$)
is related to $X\in \Delta\setminus \{ 0\} $ by  $\bar{r} - r =\sigma (X) -  \bar{\sigma}(X) $ if and only if
 \begin{equation}
 z_{i} = \bar{z}_{i},\ \mathrm{Im}\,  \lambda_{i} (X)=0,\ z^{\prime}_{j} =\bar{\beta}_{j} (X).
 \end{equation}
 In particular, $X\in \Delta_{0}.$ 
 Moreover, if  the above relation is true then
 \begin{equation}
 \langle r +\bar{r} , \mathrm{Im}\, \sigma (Y) \rangle_{\mathfrak{g}} =  2\sum_{i} z_{i} \mathrm{Im}\, \lambda_{i}(Y) +\frac{1}{2\sqrt{-1}}
 \sum_{j} \epsilon_{j}  (\bar{\beta}_{j}\wedge \beta_{j}) (X, Y),
 \end{equation}
 for any  $Y\in \Delta$. 
 Condition C)  from Corollary \ref{index-zero} 
 becomes: for any $X\in \Delta_{0} \setminus \{ 0\}$ and $z_{i}\in \mathbb{R}$ ($i=1,\ldots, p$),
 there is $Y\in \Delta \setminus \{ 0\}$ such that
  \begin{equation}\label{translate}
 \mathrm{Im}\, ( \epsilon (X, Y)) +   \frac{1}{2\sqrt{-1}}\sum_{j} \epsilon_{j}   ({\beta}_{j} \wedge \bar{\beta}_{j} )(X, Y) -
 \sum_{i} z_{i}\mathrm{Im}\, \lambda_{i} (Y)\neq 0.
 \end{equation}
 Recall the definition \eqref{omega:eq} of $\omega$. 
 Relation (\ref{translate}) shows that the map
 $$
 \omega : \Delta_{0} \rightarrow \Delta^{*},\ X \mapsto i_{X}\omega 
 $$ 
 is injective 
 (take $z_{i} :=0$ for all  $i$) and that
 \begin{equation}
 \omega (\Delta_{0}) \cap \mathrm{span}_{\mathbb{R}} \{ \mathrm{Im}\, (\lambda_{i}\vert_{\Delta}) \} = \{ 0\} .
 \end{equation}
 We obtain that  $\Delta^{*}$ decomposes into a direct sum of 
 the ($\mathrm{dim}\, \Delta - p$)-dimensional vector space 
 $\omega (\Delta_{0} )$ and  the $p$-dimensional vector space 
 $\mathrm{span}_{\mathbb{R}} \{ \mathrm{Im}\, (\lambda_{i} \vert_{\Delta})\}$. It follows that
 $$
 \Delta =  (\Delta_{0})^{\perp_\omega} \oplus \Delta_{0},
 $$
 where $(\Delta_{0})^{\perp_\omega}$ denotes  the set of vectors 
 $v\in \Delta$ which are $\omega$-orthogonal to $\Delta_{0}.$ 
 We obtain that 
 $\omega\vert_{\Delta_{0}\times \Delta_{0}} $ is non-degenerate.  
  \end{proof}

\section{The $(1,0)$-bundle of  a generalized complex structure} \label{bundle-section}

\subsection{Integrability of generalized complex structures}

Let $\mathcal J$ be a generalized almost complex structure on a  standard Courant algebroid $E = TM\oplus T^{*}M\oplus \mathcal G$  
with $(1,0)$-bundle $L$.
Then $L$  is a maximal isotropic subbundle of $E^{\mathbb{C}}$ of real index zero.
Applying the arguments from the previous sections pointwise we obtain that $L$ 
can be described
as in Proposition \ref{vspace-max-iso}
in terms of a data 
$(W,\mathcal D ,  \sigma,  \epsilon)$,  where, at any $x\in M$,  
$W_{x}\subset (T_{x}M)^{\mathbb{C}}$ is a vector subspace, 
$\mathcal D_{x}\subset \mathcal G_{x}^{\mathbb{C}}$ is a maximal isotropic subspace,
$\sigma_{x} : W_{x} \rightarrow \mathcal G_{x}^{\mathbb{C}}$ is linear  and $\epsilon_{x} 
\in \Lambda^{2} W^{*}_{x}$.  The data $(W, \mathcal D , \sigma,  \epsilon)$ 
satisfies  pointwise the conditions from  Lemma \ref{index-zero} and  Proposition \ref{basis}.
 
The smoothness of the data $(W ,\mathcal D , \sigma,  \epsilon)$ is ensured by the following result. 

 \begin{prop}\label{prop-smooth} $W=\pi_{(TM)^\mathbb{C}}(L)$ is of constant rank if and only if $W\cap \bar W$ is. In that case,  $W\subset (TM)^{\mathbb{C}}$,  
 $W\cap \bar W$  
 and 
$\mathcal D\subset\mathcal G^{\mathbb{C}}$   are smooth subbundles 
 and
 the 
possibly discontinuous sections 
 $\sigma, \epsilon$ of $\mathrm{Hom}(W,\mathcal G^{\mathbb{C}})$ and $\Lambda^2W^*$  respectively, can be chosen to be smooth. 
\end{prop}
\begin{proof} The first claim follows from 
$\mathrm{dim}\, (\Delta_{x}) = 2 \mathrm{dim}_{\mathbb{C}} (W_{x}) - \mathrm{dim}\, (M)$, see  \ Lemma \ref{index-zero}  A).
For any $x_0\in M$ there exist $\dim W_{x_0}$ smooth sections of $L$ projecting to a basis of $W_{x_0}$. Since $W$ has constant rank,  the projections of those sections form a smooth frame of $W$ in a neighborhood of $x_0$. This proves that 
$W\subset (TM)^{\mathbb{C}}$ and  $\ker \pi_{(TM)^\mathbb{C}}|_L\subset L$  are smooth subbundles and 
$\mathrm{rank} (\ker \pi_{(TM)^\mathbb{C}}|_L)=\mathrm{rank}\, L - \mathrm{rank}\, W$.  The smoothness of $W\cap \bar W$ follows easily from the fact that 
$W, \bar W\subset (TM)^{\mathbb{C}}$ are smooth subbundles with intersection of constant rank.
Since  
$\mathcal D_{x}$ is a maximal isotropic subspace of $(\mathcal G_{x})^{\mathbb{C}}$, its dimension is independent of $x\in M$ and a similar argument shows that 
$\mathcal D =   \pi_{\mathcal G^\mathbb{C}}(\ker \pi_{(TM)^\mathbb{C}} |_L)$ is a smooth  subbundle of $\mathcal G^{\mathbb{C}}$.
If we choose a subbundle $\mathcal K\subset L$ complementary to $\ker \pi_{(TM)^\mathbb{C}}|_L$ and define 
$\sigma :=\pi_{\mathcal G}^{\mathbb{C}} \circ (\pi_{(TM)^\mathbb{C}}|_{\mathcal K})^{-1}$, then $\sigma$ is smooth. 
The smoothness of 
$\epsilon \in \Gamma ( \Lambda^{2} W^{*})$ follows from the way it is determined by the smooth data $\sigma$ and $L$, 
see the proof of Proposition~\ref{vspace-max-iso}.
\end{proof}

\begin{ass}{\rm \label{Wsmooth:ass}
From now on    we assume that $W$ is of constant rank  and that $\sigma$ and $\epsilon$ are smooth. 
}
\end{ass}

\begin{rem}{\rm In view of Corollary \ref{complement_sigma:eq}, choosing a subbundle $\mathcal C \subset \mathcal G^\mathbb{C}$ complementary to $\mathcal D$
we can modify any smooth section $\sigma$ of $\mathrm{Hom}(W,\mathcal G^\mathbb{C})$ by the affine action 
described in Proposition~\ref{action_on_Sd:prop} 
such as to take values 
in $\mathcal C$. It suffices to apply $T_{\gamma}$ for $\gamma = -\pi_{\mathcal D} \circ \sigma$. Here $\pi_{\mathcal D} :
\mathcal G^{\mathbb{C}}  = \mathcal D \oplus \mathcal C\to \mathcal D$ stands for the projection defined by $\mathcal C$. The resulting
section   $\sigma + \gamma$  is again smooth. 
}
\end{rem}

\begin{prop}\label{conditii-integr}  
Let  $[\cdot  , \cdot ]_{\mathcal G}$ and $\langle \cdot , \cdot \rangle_{\mathcal G}$ be the Lie bracket and scalar product of 
the quadratic Lie algebra bundle  $\mathcal G$ of $E$ and 
$(\nabla , R , H)$  the defining data for $E$. 
A  generalized almost complex structure $\mathcal J$ on $E$
with $(1,0)$-bundle $L = L (W,\mathcal D ,  \sigma,  \epsilon)$ 
 is integrable if and only if the following conditions 
hold:\

A)  $W$ is an involutive distribution on $M$;\

B) $\mathcal D$ is a Lagrangian subbundle  of $\mathcal G^{\mathbb{C}}$ (i.e. $\mathcal D_{x}\subset \mathcal G^{\mathbb{C}}_{x}$ is a Lagrangian subalgebra,
for any $x\in M$);

C) the partial connection on $\mathcal G^{\mathbb{C}}$ 
$$
(\nabla + \mathrm{ad}_{\sigma})_{X} r:= \nabla_{X} r + [\sigma(X), r]_{\mathcal G},\ \forall X\in \Gamma (W),\ r\in 
\Gamma (\mathcal G^{\mathbb{C}}), 
$$
preserves $\mathcal D$;\

D) for any $X, Y \in W$, 
\begin{equation}\label{r}
R(X, Y) +  (d^{\nabla} \sigma )(X, Y) + [ \sigma(X), \sigma (Y) ]_{\mathcal G } \in \mathcal D;
\end{equation}
E) on $\Lambda^{3}W$,
\begin{equation}\label{d-epsilon}
 H + 2 d\epsilon  +2 \langle R\wedge \sigma \rangle_{\mathcal G} +  \langle  (d^{\nabla}\sigma )\wedge \sigma \rangle_{\mathcal{G}} + 
  2 \langle [\sigma, \sigma ], \sigma \rangle_{\mathcal G} =0,
\end{equation}
where 
$$
\langle [\sigma, \sigma ], \sigma \rangle_{\mathcal G}(X, Y, Z) := \langle [\sigma (X), \sigma(Y) ]_{\mathcal G}, \sigma(Z) \rangle_{\mathcal G},\ \forall X, Y, Z\in W.
$$
\end{prop}

\begin{proof} From relations (\ref{dorfman1}) and (\ref{dorfman2}),   for any sections $u = X+\xi +\sigma (X) + r$ and 
$v= Y+\eta +\sigma (Y) +s $ of $L$,  the projections of the Dorfman bracket $[u, v]$ onto the components of $E^{\mathbb{C}}$ are  given by 
\begin{align*}
\nonumber \pi_{(TM)^{\mathbb{C}}}  [u, v] &= \mathcal L_{X}Y\\
\nonumber \pi_{(T^{*}M)^{\mathbb{C}}}  [u, v]& = 
i_{Y} i_{X} H +\mathcal L_{X} \eta   -  i_{Y} d\xi- 2\langle i_{X}R, \sigma (Y) +s\rangle \\
\nonumber& + 2\langle i_{Y} R, \sigma (X) +r\rangle + 2\langle \nabla ( \sigma (X) + r),
\sigma(Y) +s \rangle\\
\nonumber   \pi_{\mathcal G^{\mathbb{C}}} [u, v]&= R(X, Y) +\nabla_{X} (\sigma (Y) +s ) -\nabla_{Y} (\sigma(X) + r)\\
\nonumber& + [\sigma (X) + r, \sigma(Y) +s]_{\mathcal G},
\end{align*}
where, in order  to simplify notation, we write  $\langle \cdot , \cdot \rangle$ instead of $\langle \cdot , \cdot \rangle_{\mathcal G}$.
It follows  that $L$ is integrable if and only if $W$ is involutive,  for any $X, Y\in \Gamma (W)$,
\begin{align}
\nonumber& R(X, Y) +\nabla_{X} ( \sigma (Y) +s) -\nabla_{Y} ( \sigma(X) + r)\\
\label{cond-1}&  + [\sigma (X) + r, \sigma (Y) +s]_{\mathcal G} -\sigma ( \mathcal L_{X} Y ) := r_{1}  \in \Gamma (\mathcal D)
\end{align}
and, for any $V\in \Gamma ( W)$,
\begin{align}
\nonumber& (i_{Y} i_{X} H)(V) + (\mathcal L_{X}\eta )(V) -  (i_{Y}d\xi )(V)  - 2\langle R(X, V), \sigma (Y) +s \rangle \\
\nonumber& + 2\langle R(Y, V), \sigma (X) + r\rangle + 2\langle \nabla_{V} (\sigma (X) + r), \sigma (Y)+s \rangle\\
\label{cond-2}&  = 2 \epsilon ( \mathcal L_{X}Y, V) -\langle \sigma ({ \mathcal L}_{X} Y), \sigma (V) \rangle
 - 2 \langle r_{1}, \sigma (V)\rangle .
 \end{align}
 Relation (\ref{cond-1})  with $X= Y =0$ implies that 
 $\mathcal D$  a Lie algebra subbundle  (and,  being maximal isotropic, a Lagrangian subbundle) of $\mathcal G^\mathbb{C}$. Taking 
 $X=0$, $s=0$ and $r =s =0$ we obtain conditions C) and D). 
 Using the definition of  $r_{1}$,   by a calculation eliminating $\xi$ and $\eta$ we obtain  that  (\ref{cond-2}) is equivalent to 
\begin{align}
\nonumber&  H(X, Y, V) + 2d\epsilon (X, Y, V) + \langle  (d^{\nabla}\sigma) \wedge \sigma\rangle (X, Y, V) + 2\langle R\wedge \sigma \rangle (X, Y, V)\\
 \nonumber& - 2\langle s,  R( X, V)+   (d^{\nabla} \sigma )(X, V)+ [\sigma (X), \sigma (V) ]_{\mathcal G}  \rangle\\
 \nonumber& + 2 \langle {r},   R( Y, V)+  (d^{\nabla} \sigma )(Y, V)  + [\sigma (Y)), \sigma (V) ]_{\mathcal G}  \rangle\\
 \label{cond-3}&  + 2 \langle  \nabla_{V}r, s \rangle +  2\langle [ \sigma (X), \sigma (Y)]_{\mathcal G} +  [r, s ]_{\mathcal G}, \sigma (V) \rangle
  =0.
 \end{align}
 From  condition D), the second and third lines in (\ref{cond-3}) vanish, as $\mathcal D$ is isotropic. 
 Using  that $\langle \cdot , \cdot \rangle_{\mathcal  G}$ is $\mathrm{ad}$-invariant
 and condition C),  we obtain that
 \begin{equation}
 \langle \nabla_{V} r, s \rangle  + \langle [r, s ]_{\mathcal G},  \sigma (V)\rangle  =0,
  \end{equation}
  i.e.\ the last line in   (\ref{cond-3}) reduces to 
 $2\langle [ \sigma (X), \sigma (Y)]_{\mathcal G} , \sigma (V) \rangle$.   
  Relation (\ref{cond-3}) simplifies to 
(\ref{d-epsilon}). 
\end{proof}

\subsection{Generalized complex structures and equivalences}

Let  $I: E_{1} \rightarrow E_{2}$ be an isomorphism between standard Courant algebroids.  It
induces naturally a map between the sets of generalized complex structures of $E_{i}$ and, equivalently, 
a map $I_{\mathfrak S}:  {\mathfrak S}(E_{1}) \rightarrow \mathfrak{S} (E_{2}) $ 
between the sets  of involutive maximal  
isotropic subbundles $E_{i}^\mathbb{C}$ of real index zero. 
Let
\[ \mathcal L_{E_{i}} : \mathcal S(E_{i}) \to \mathfrak S(E_{i}),\quad (W, \mathcal D , \sigma ,  \epsilon )\mapsto L(W,\mathcal D , \sigma  , \epsilon),\] 
where $\mathcal S(E_{i})$ is the set of data defining involutive maximal isotropic subbundles $L\subset E_{i}^\mathbb{C}$ of real index zero.

\begin{prop}\label{lifted_action:prop} 
i) Define the map $T_I : \mathcal S(E_{1})  \to \mathcal S(E_{2}) $ by 
\[  (W,\mathcal D , \sigma ,  \epsilon ) \mapsto  (W,K \mathcal D ,
K\sigma +\Phi,  \epsilon +\frac12(\beta\vert_{W} + \langle \Phi \wedge K\sigma \rangle_{\mathcal G_{2}})),\]
where $(K,\Phi,\beta)$ are the data induced by $I$. 
Then  $T_I$ 
is  bijective and is a lift of $I_{\mathfrak{S}}$, i.e.\ it satisfies $\mathcal L_{E_{2}} \circ T_{I} = I_{\mathfrak S} \circ \mathcal L_{E_{1}}$. 
We  call  it  the  {\cmssl canonical lift  of $I_\mathfrak S$}.\

ii) When $E_{1} = E_{2}:= E$,   the map $I\mapsto T_I$ is an action of  $\mathrm{Aut}\, (E)$ on ${\mathcal S}(E)$
lifting the natural action of $\mathrm{Aut}\, (E)$ on $\mathfrak S(E)$.
\end{prop}

\begin{proof}  i) The first statement is straightforward. 

ii) In terms of  the components  $(K,\Phi,\beta)$ of  $I\in \mathrm{Aut}\,  (E)$, the group
multiplication in $\mathrm{Aut}\, (E)$ is given by
\begin{equation}\label{gplaw:q} (\tilde{K},\tilde\Phi ,\tilde \beta) \cdot (K,\Phi ,\beta) = (\tilde K K,\tilde K \Phi +\tilde \Phi, \beta +\tilde \beta -2(\tilde \Phi^*\tilde K\Phi )^{\mathrm{sk}}),\end{equation}
where $(\tilde \Phi^*\tilde K\Phi )^{\mathrm{sk}}$ denotes the skew part of the bilinear form $\tilde \Phi^*\tilde K\Phi =\langle K\Phi , \tilde \Phi\rangle$ on $W$
(see  \cite[Proposition 6]{cortes-david-JSG}). 
Using the above expression for the group law 
we can check that  the inversion in $\mathrm{Aut}\, (E)$  is given by
\[ (K,\Phi ,\beta) \mapsto (K^{-1},-K^{-1}\Phi, -\beta ).\] 
This implies that $T_{I^{-1}}$ is given by 
\begin{equation} 
(W, \mathcal D , \sigma ,  \epsilon )\mapsto (W,K^{-1} \mathcal D,  K^{-1}(\sigma -\Phi ), \epsilon -\frac12(\beta\vert_{W}  + \langle \Phi \wedge \sigma \rangle_{\mathfrak g})),
\end{equation}
where we used $\langle  K^{-1}\Phi \wedge K^{-1}\sigma \rangle_{\mathfrak g} = \langle \Phi \wedge \sigma \rangle_{\mathfrak g}.$ 
Applying $T_I$ to the result yields 
\[ (W, \mathcal D , \sigma , (\epsilon -\frac12(\beta\vert_{W} + \langle \Phi \wedge \sigma \rangle_{\mathfrak g}) 
+\frac12 (\beta\vert_{W} + \langle (\Phi \wedge K(K^{-1}(\sigma -\Phi) )\rangle_{\mathfrak{g}})) = (W,\mathcal D , \sigma, \epsilon).\]
This proves  that $T_I$ is a bijection and $(T_I)^{-1}= T_{I^{-1}}$. 
Using \eqref{gplaw:q} and the definition of $T_I$ it is easy to check that $I\mapsto T_I$ is a group 
homomorphism. 
\end{proof}

\begin{cor}\label{LI:cor}Let $I: E_{1} \rightarrow E_{2}$ 
be a Courant algebroid isomorphism between standard Courant algebroids and 
$I_{\mathfrak S} : {\mathfrak S} (E_{1}) \rightarrow {\mathfrak S}(E_{2})$ the induced map. 
Then any lift $L_{I} : \mathcal S (E_{1}) \rightarrow \mathcal S (E_{2})$ of $I_{\mathfrak S}$ is of the form $L_{I} = L_{I, \gamma}$, where  
$\gamma  \in \Gamma(\mathrm{Hom}(W,\mathcal D))$ and 
\begin{align} 
\nonumber&  L_{I,\gamma} (W, \mathcal D , \sigma ,  \epsilon ) :=\\
\nonumber&  \left( W,K \mathcal D, 
K(\sigma +\gamma ) + \Phi , \epsilon +\frac12\left(\beta\vert_{W} + \langle (K^{-1}\Phi -\gamma )\wedge \sigma \rangle
 +\langle K^{-1}\Phi \wedge \gamma\rangle\right) \right) .
\end{align}
Above $(K,\Phi,\beta)$ are the components of $I$, and, for simplicity,
$\langle \cdot , \cdot \rangle_{\mathcal G_{2}}$ was denoted by $\langle \cdot , \cdot \rangle$.
\end{cor}

\begin{proof}This follows from Propositions \ref{lifted_action:prop} and \ref{action_on_Sd:prop}.
\end{proof}
Let $f: M \rightarrow N$ be a diffeomorphism and $E$ a standard Courant algebroid over $N$, with quadratic Lie algebra bundle
$(\mathcal G , [\cdot , \cdot ]_{\mathcal G},  \langle \cdot , \cdot \rangle_{\mathcal G}).$ 
Then $f$ induces naturally a map between the sets of generalized complex structures on $E$ and  the pullback $f^{!}E$,  which assigns to 
a generalized complex structure $\mathcal J$ on $E$ the generalized complex structure $f^{!}\circ \mathcal J \circ (f^{!})^{-1}$.  Let $f_{\mathfrak S} :\mathfrak{S} (E) \rightarrow \mathfrak{S} ( f^{!}E)$ be the induced map on the sets of involutive maximal isotropic subbundles  of 
$E^{\mathbb{C}}$ and $(f^{!}E)^{\mathbb{C}}$ of real index zero. 
The next proposition is analogous to Proposition 
\ref{lifted_action:prop} and Corollary \ref{LI:cor} and can be checked easily. 

\begin{prop}\label{sigma0-isomorphism} i) Define the map $T_{f} : \mathcal S (E) \rightarrow \mathcal S (f^{!}E)$ by 
\begin{equation}
(W, \mathcal D , \sigma ,  \epsilon ) \mapsto  (  (df)^{-1}(W_f),   f^{*}\mathcal D ,f^{*}\sigma , f^{*}\epsilon )
\end{equation}
where $W_f= W \circ f$ and
\begin{align}
\nonumber& (f^{*} \sigma) (X_{x}) = \sigma ( df ( X_{x})) \in \mathcal G^{\mathbb{C}}_{f(x)} = (f^{*} \mathcal G^{\mathbb{C}} )_{x},\\\nonumber&  (f^{*} \mathcal D)_{x}:=
 \mathcal D_{ f(x)}\subset  \mathcal G^{\mathbb{C}}_{f(x)} = (f^{*} \mathcal G^{\mathbb{C}} )_{x}\\
\nonumber&  (f^{*}\epsilon )(X_{x}, Y_{x}) = \epsilon ( df (X_{x}), df(Y_{x})),
\end{align}
for any  $x\in M$,  $X_{x}, Y_{x}\in (df)^{-1}W_{f(x)}\subset T_{x}M$. Then $T_{f}$ is a bijection and satisfies 
${\mathcal L}_{f^{!} E} \circ T_{f} = f_{\mathfrak S} \circ \mathcal L_{E}$.\ 

ii) Any other lift of $f_{\mathfrak S}$ is of the form 
\begin{equation}
(W, \mathcal D , \sigma ,  \epsilon ) \mapsto  (  (df)^{-1}(W_f), f^{*}\mathcal D , 
f^{*}\sigma  +\gamma , f^{*}\epsilon + \frac{1}{2} \langle  f^{*}\sigma\wedge \gamma 
\rangle_{\mathcal G}), 
\end{equation}
where   $\gamma \in \Gamma (\mathrm{Hom} ( (df)^{-1} (W_{f}), f^{*} \mathcal D ).$\

iii)  When $M = N$ and $E$ is untwisted, the map $f\mapsto T_{f}$ is an action of $\mathrm{Diff} (M)$ on $\mathcal S (E)$
lifting the natural action of $\mathrm{Diff}(M)$ on 
$\mathfrak{S} (E)$.  
\end{prop}

\section{Proof of our main result}\label{classif-local}

Let  $\mathcal J$ be a generalized complex structure on  a transitive Courant algebroid $E$  over a manifold $M$ with $(1,0)$-bundle $L = L( W,\mathcal D ,  \sigma,  \epsilon  )$.

\begin{ass}{\rm   The fiber of the quadratic Lie algebra bundle of $E$ is a compact semisimple Lie algebra.}
\end{ass}

\begin{defn}\label{norm}   i) A point $x\in M$ is called {\cmssl regular} for $\mathcal J$ if the dimension of $W_x$ (or equivalently that of $\Delta_x :=(W_x\cap \bar W_x)_{\mathbb R}$, cf.\ Proposition \ref{prop-smooth})  is  constant in a neighbourhood of $x$.\

ii) The generalized complex structure  $\mathcal J$
is called {\cmssl regular}   if  for any $x\in M$ there is an open  neighbourhood $U$ of $x$ in $M$ and an isomorphism $I_{U}: E\vert_{U} \rightarrow E_{U}$ 
where $E_{U}: = TU\oplus T^{*}U\oplus  (U\times \mathfrak{g})$  is an untwisted Courant algebroid which maps $\mathcal D$ to a trivial bundle
$U\times {\mathcal D}_{U}$, where  ${\mathcal D}_{U}\subset \mathfrak{g}^{\mathbb{C}}$ is a regular subalgebra of $\mathfrak{g}^{\mathbb{C}}$
(that is, $\mathcal D_{U}$ is normalised by a Cartan subalgebra of $\mathfrak{g}$). 
\end{defn}

\begin{rem}\label{wang-rem}{\rm  When  $M$ is a point,  $E$  is a compact semisimple quadratic Lie algebra 
$(\mathfrak{g}, [\cdot , \cdot ]_{\mathfrak{g}}, 
\langle \cdot , \cdot \rangle_{\mathfrak{g}})$
and  $\mathcal J$  defines a  skew-symmetric left-invariant complex structure on  a Lie group  $G$ with Lie 
algebra 
$\mathfrak{g}$. From  \cite{wang}, $\mathcal J$  is regular. This motivates  the  regularity  assumption on  the fibers of
$\mathcal D .$ 
On the other hand, from Section \ref{standard-untwisted-section}  any transitive Courant algebroid is locally isomorphic to an untwisted Courant algebroid. 
The quadratic Lie algebra bundle of the latter is trivial and the Lie bracket and scalar product have constant coefficients in a 
constant trivialisation.  Therefore, it is also natural to assume that such  local isomorphisms map $\mathcal D$ to  trivial subbundles
of $M\times \mathfrak{g}^{\mathbb{C}}.$}
\end{rem}

The remaining part of this section is devoted to the proof of   Theorem \ref{main-thm}.   
Let $\mathcal J$ be a regular generalized complex structure on a transitive Courant algebroid $E$. We denote by 
$(\mathfrak{g},[\cdot , \cdot ]_{\mathfrak{g}}, \langle \cdot , \cdot \rangle_{\mathfrak{g}})$  the fiber-type 
of the quadratic Lie algebra bundle of $E$ and we assume that $\mathfrak{g}$ is a compact  semisimple Lie algebra.
Without loss of generality, we can (and will) assume that $E = TM\oplus T^{*}M\oplus  \mathcal G$ (where $\mathcal G := M\times \mathfrak{g}$)
is  untwisted and that $\mathcal D$ is a trivial subbundle of $\mathcal G^{\mathbb{C}}.$ 
As before, we denote by  
$\nabla$ the  canonical connection on $\mathcal G$.  We identify (and denote by the same symbols) 
subspaces  or  vectors from $\mathfrak{g}^{\mathbb{C}}$ with  the corresponding constant subbundles
or sections 
of $\mathcal G^{\mathbb{C}}.$ 
We will apply   successively  various (local) equivalences to    $\mathcal J$ in a neighborhood of a regular point.
From the regularity of  ${\mathcal D}$, 
\begin{equation}\label{k-d-redenoted}
\mathcal D =  \mathfrak{h}^{\mathcal D} +\mathfrak{g}^{\mathbb{C}}(R_{0})
\end{equation}
where   $R_{0}\subset R$ is a 
closed subset of the root system  $R$ 
of $\mathfrak{g}^{\mathbb{C}}$ 
relative to the complexification 
$\mathfrak{h} =( \mathfrak{h}_{\mathfrak{g}})^{\mathbb{C}}$  of  a Cartan subalgebra  $\mathfrak{h}_{\mathfrak{g}}$ of $\mathfrak{g}$ 
(that is, if $\alpha , \beta \in R_{0}$ are such that $\alpha +\beta \in R$, then $\alpha +\beta \in R_{0}$),
$$
\mathfrak{g}^{\mathbb{C}} (R_{0}) := \sum_{\alpha \in R_{0}} \mathfrak{g}_{\alpha}
$$
where $\mathfrak{g}_{\alpha}$ denotes the root space corresponding to the root $\alpha$ 
and $\mathfrak{h}^{\mathcal D} := \mathcal D \cap \mathfrak{h}.$ 
Relation   (\ref{k-d-redenoted})
follows  from the fact that $\mathfrak{h}$ is the complexification 
of a compact torus. It is also a special case of the more general 
statement (4.2) of \cite{wang}.

\begin{lem}\label{r0}  The set of roots $R_{0}$ is a positive root system  of $R$. 
\end{lem}

\begin{proof}   Since  $R_{0}$ is closed, 
it remains  to show that  
$R_{0}^{\mathrm{sym}} := R_{0} \cap (- R_{0}) =\emptyset$ and 
$R_{0} \cup ( - R_{0}) = R$ ( \cite[Corollary 1, page 161]{bourbaki}). 
Since $\langle \cdot , \cdot \rangle_{\mathfrak{g} }$ is $\mathrm{ad}$-invariant and non-degenerate, any root vectors $E_\alpha\in \mathfrak{g}_{\alpha}\setminus \{ 0\}$ and root spaces $\mathfrak{g}_\alpha$ satisfy
\begin{align}
\nonumber& \langle E_{\alpha}, E_{-\alpha }\rangle_{\mathfrak{g} } \neq 0 ,\ \alpha ( [E_{\alpha}, E_{-\alpha}]_{\mathfrak{g} }) \neq 0,\ \forall \alpha \in R\\
\label{knapp-rel}& \langle \mathfrak{g}_{\alpha}, \mathfrak{g}_{\beta} \rangle_{\mathfrak{g} } =0,\ \forall \alpha , \beta \in R\cup \{ 0\} ,\ \alpha +\beta \neq 0,
\end{align} 
where $\mathfrak g_0 := \mathfrak h$ (see \cite[Proposition 2.17, pages 95, 97]{knapp}).  
Since  $\mathcal D$ isotropic,
the first relation in (\ref{knapp-rel}) 
implies  that $R_{0}^{\mathrm{sym}} = \emptyset .$ 
Assume, by contradiction, that  there is $\alpha \in R\setminus  (R_{0} \cup  (- R_{0})).$ 
Since $\alpha \notin R_{0}$, $\mathcal D$ is strictly included in $\mathcal D + \mathfrak{g}_{\alpha }.$ Since $\alpha \notin 
(- R_{0})$, the third relation in  (\ref{knapp-rel}) implies that $\mathcal D + \mathfrak{g}_{\alpha }$ is isotropic. 
But  $\mathcal D$ is maximal isotropic.  We obtain a contradiction.
\end{proof}

From relations  (\ref{k-d-redenoted}) and  (\ref{knapp-rel}),  $\mathfrak{h}^{\mathcal D}\subset \mathfrak{h}$ is maximal isotropic. 
Let 
\begin{equation}\label{triv-h}
\mathcal B ({\mathfrak{h}} ):= \{ v_{1}, \ldots , v_{p},  w_{1}, \ldots w_{q},
{v}^{\prime}_{1}, \ldots  , {v}^{\prime}_{p},  \tau (w_{1}), \ldots ,
\tau (w_{q}) \}
\end{equation}
be a  basis of  $\mathfrak{h}$  
adapted to  $\mathfrak{h}^{\mathcal D}$  (see Proposition \ref{max-iso-chevalley}), where $\tau : \mathfrak{g}^{\mathbb{C}} \rightarrow \mathfrak{g}^{\mathbb{C}}$
is the anti-involution determined by $\mathfrak{g}.$ 
Since $R_{0}= R^{+}$ is a positive root system, 
$$
\mathcal C = \mathfrak{h}^{\mathcal C} + \mathfrak{g}^{\mathbb{C}} (R^{-})
$$
is a   complement of $\mathcal D$  in
$\mathfrak{g}^{\mathbb{C}}$, where $R^{-} := - R^{+}$ and   
$\mathfrak{h}^{\mathcal C}$ is generated by   $\{ {v}^{\prime}_{i}, \tau (w_{j})\}$. 
From Corollary~\ref{complement_sigma:eq},   we can assume that $\sigma$ takes values in 
$\mathcal C$.
Condition C) from  Proposition \ref{conditii-integr}  implies that $\sigma (X)$ normalises 
$\mathcal D$  for any $X\in W$  
(since $\nabla$ is the canonical connection on $M\times \mathfrak{g}$ and 
$\mathcal D\subset M\times \mathfrak{g}^{\mathbb{C}}$ is trivial).
Then, a simple argument shows that 
$\sigma : W \rightarrow \mathfrak{h}^{\mathcal C} .$ We write
\begin{equation}\label{sigma-prel-reg}
\sigma(X) = \sum_{i=1}^{p} a_{i}(X) {v}^{\prime}_{i} + \sum_{j=1}^{q}  b_{j}(X) \tau (w_{j}),\ \forall X\in W,
\end{equation}
where $a_{i}, b_{j}\in \Gamma (W^{*}).$ 
Using that $\nabla {v}^{\prime}_{i} = \nabla  (\tau (w_{i}) )=0$, 
relations (\ref{r}) and (\ref{d-epsilon}) 
imply  that $a_{i}, b_{j}$  and $\epsilon$ are closed  forms. 
Restricting $M$ if necessary, 
we can find $f_{i}, g_{j} \in C^{\infty}(M, \mathbb{C})$ 
such that
$a_{i} = d f_{i}\vert_{W}$ and 
$b_{j} = d g_{j}\vert_{W}$.
To summarise our argument so far:

\begin{lem}\label{step-two} Any  regular generalized complex structure is  equivalent, in a neighborhood of a regular point, to a generalized complex structure 
$\mathcal J$ defined on an untwisted Courant algebroid $TM\oplus T^{*}M \oplus (M\times \mathfrak{g})$,   
with $(1,0)$-bundle $L = L(W, \mathcal D , \sigma  , \epsilon )$, where 
\begin{align}
\nonumber&\mathcal D = \mathrm{span}_{\mathbb{C}} \{ v_{1}, \ldots , v_{p}, w_{1} , \ldots ,  w_{q} \} + \mathfrak{g}^\mathbb{C} (R^{+}) \\
\label{w-d}&    \sigma = \sum_{i=1}^{p} (df_{i})\vert_{W}\otimes  {v}^{\prime}_{i} + \sum_{j=1}^{q} (d g_{j})\vert_{W} \otimes  \tau (w_{j}), 
\end{align}
where $f_{i}, g_{j}\in C^{\infty}(M, \mathbb{C})$ and 
\begin{equation}\label{sigma-e}
d\epsilon =0\quad \mathrm{on}\, \Lambda^{3}W.
\end{equation}
When $\Delta =  (W\cap \bar{W})_{\mathbb{R}}\neq \{ 0\}$,
the $1$-forms $\{ d  ( \mathrm{Im}\, f_{i} )\vert_{\Delta} \}$ are linearly independent and 
\begin{equation}
\omega:= \mathrm{Im}\,  \epsilon +  \frac{1}{2\sqrt{-1}}  \sum_{j=1}^{q} \epsilon_{j} dg_{j}\wedge d\bar{g}_{j}
\end{equation}
is non-degenerate on 
$$
\Delta_{0}:= \{ X\in \Delta \mid  d (\mathrm{Im} f_{i} )(X) =0,\ 1\leq i\leq p\} .
$$
When $\Delta =\{ 0\}$,
$q =\frac{1}{2} \mathrm{dim}\, (\mathfrak{h})$,
$\mathcal D\cap \bar{\mathcal D} =\{ 0\}$ and
$\sigma = \sum_{j=1}^{q} (dg_{j}) \vert_{W}\otimes \tau (w_{j})$. 
\end{lem}

\begin{proof} 
It remains to explain the linear independence of  $\{ d  ( \mathrm{Im}\, f_{i} )\vert_{\Delta} \}$ and the non-degeneracy of $\omega\vert_{\Delta_{0}}.$ 
For any $\alpha \in R$,  let  $E_{\alpha}\in \mathfrak{g}_{\alpha}\setminus \{ 0\}$ such that $\tau (E_{\alpha } ) =  - E_{-\alpha}$ 
 and $\langle E_{\alpha }, \tau (E_{\alpha }) \rangle_{\mathfrak{g}} =\epsilon_{\alpha}$, where $\epsilon_{\alpha}\in \{ \pm 1\}$.
Completing  the  basis 
$$
{v}_{1}, \ldots , {v}_{p},  w_{1}, \ldots  ,  w_{q},\, E_{\alpha}\  (\alpha \in R^{+})
$$
of $\mathcal D$ with 
\begin{equation}
{v}^{\prime}_{1}, \ldots , {v}^{\prime}_{p}, \tau (w_{1}), \ldots  , \tau (w_{q}),\, E_{\beta}\  (\beta \in R^{-})
\end{equation}
we obtain a  basis  of $\mathfrak{g}^{\mathbb{C}}$ adapted to $\mathcal D .$  
Using this basis,  we conclude  from 
Proposition \ref{basis}. 
\end{proof}

\begin{lem}\label{cor-unu}  Let $\mathcal J$ be a  generalized complex structure 
with $(1,0)$-bundle $L = L(W, \mathcal D , \sigma,  \epsilon)$
as in  Lemma  \ref{step-two}.

i) There is $I\in \mathrm{Aut}  (E)$ 
with $K$-component acting as the identity on $M\times \mathfrak{h}_{\mathfrak{g}}$ and such that 
the $(1,0)$-bundle of $I(\mathcal J)$ is given by 
$\tilde{L} = L( \tilde{W}, \tilde{ \mathcal D}, \tilde{\sigma},\tilde{\epsilon})$
where  
$\tilde{W} = W$, 
$\tilde{\mathcal D} = \mathcal D$  and
\begin{equation}\label{sigma}
\tilde{\sigma} (X) \equiv \sqrt{-1} \sum_{j=1}^{p} X(f_{j}) {v}^{\prime}_{j},\ \forall X\in W,
\end{equation}
where  $f_{j}\in C^{\infty}(M, \mathbb{R})$ and the symbol $\equiv$ denotes  equality modulo $\mathcal D .$\

ii)  In  a neighborhood of a regular point,  $\mathcal J$ is equivalent to a generalized complex structure as in Lemma \ref{step-two}, with $f_{i}$ ($1\leq i\leq p$) purely imaginary  and $g_{j} =0$ ($1\leq j\leq q$). 
For such a generalized complex structure  $(df_{i}\vert_{\Delta})$ are linearly independent and 
 $\mathrm{Im}\, \epsilon$ is non-degenerate on $\bigcap_{1\leq i\leq p} \ker (df_{i}\vert_{\Delta})$ when $\Delta\neq \{ 0\}.$
 \end{lem}

\begin{proof} i) Let   $K\in \mathrm{Aut} (M\times \mathfrak{g}  )$  be 
as in Lemma 
\ref{dec-iso}  ii) and  $I\in \mathrm{Aut} (E)$ which extends $K$ 
(according to  Lemma 
\ref{iso-untwisted}). Let  $\tilde{L} = L( \tilde{W}, \tilde{ \mathcal D},  \tilde{\sigma},\tilde{\epsilon})$ be the $(1,0)$-bundle of $I(\mathcal J )$, 
where $(\tilde{W}, \tilde{\mathcal D}, \tilde{\sigma}, \tilde{\epsilon }) = T_{I} (W, \mathcal D , \sigma , \epsilon )$.
Then
$\tilde{W} = W$, 
$\tilde{\mathcal D} = \mathcal D$  
and  we claim that 
\begin{equation}
\label{tilda-sigma0} \tilde{\sigma}  (X)= \sigma(X)  - \sqrt{-1} \sum_{j=1}^{\ell} X(\theta_{\alpha_{j}}) \tilde{H}_{\alpha_{j}},\ \forall X\in W.
\end{equation}
In the above relation  $\Pi = \{ \alpha_{1}, \ldots , \alpha_{\ell}\}$   are the simple roots from $R^{+}$
and  $\tilde{H}_{\alpha_{i}}\in \mathfrak{h}$ satisfy 
$\alpha_{j} ( \tilde{H}_{\alpha_{i}}) =\delta_{ij}$ for any $j$. In particular,    $\{ \sqrt{-1} \tilde{H}_{\alpha_{1}}, \ldots , \sqrt{-1} \tilde{H}_{\alpha_{\ell}}\}$ is a 
basis of $\mathfrak{h}_{\mathfrak{g} }.$\

In order to prove  (\ref{tilda-sigma0}), note 
that the component $\Phi\in \Omega^{1}(M, \mathfrak{g} )$ of  $I$  
satisfies 
\begin{equation}\label{phi-applied}
\langle \Phi (X), [ r,  s ]_{\mathfrak{g} } \rangle_{\mathfrak{g} }= \langle \nabla_{X} ( K^{-1} r), K^{-1} s \rangle_{\mathfrak{g} }
\end{equation} 
for any  $r,  s\in\mathfrak{g}^{\mathbb{C}} $
(from the first relation in (\ref{def-cond}) and the $\mathrm{ad}$-invariance of $\langle \cdot , \cdot \rangle_\mathfrak{g}$).
Relation  (\ref{phi-applied})  with  $r:= H\in \mathfrak{h}$ implies that
\begin{equation}\label{unu}
\Phi (X) \in \mathfrak{h}_{\mathfrak{g} },\  \forall X\in TM,
\end{equation}
where we used  $K\vert_{M\times \mathfrak{h}} = \mathrm{Id}$,  $H$  constant
and  the third relation in  (\ref{knapp-rel}). We choose again root vectors $(E_{\alpha})_{\alpha \in R}$ 
as in the proof of Lemma \ref{step-two}.
Relation (\ref{phi-applied})  with  $r:= E_{\alpha}$ and  $s:= E_{-\alpha}$  becomes
\begin{equation}
\langle \Phi (X), [ E_{\alpha}, E_{-\alpha} ]_{\mathfrak{g} } \rangle_{\mathfrak{g} } = -   \sqrt{-1}  X(\theta_{\alpha})
\langle E_{\alpha}, E_{-\alpha}\rangle_{\mathfrak{g} },
\end{equation}
for any $ \alpha \in R$, 
or
\begin{equation}\label{doi}
\alpha (  \Phi (X)) = -  \sqrt{-1}  X(\theta_{\alpha}),\  \forall \alpha \in R,
\end{equation}
where we  used that  $[ E_{\alpha}, E_{-\alpha} ]_{\mathfrak{g} } = \langle E_{\alpha}, E_{-\alpha }\rangle_{\mathfrak{g}} H_{\alpha}$
and  $\langle E_{\alpha}, E_{-\alpha}\rangle_{\mathfrak{g}} \neq 0$ 
(see  \cite[Lemma 2.18, page 95]{knapp}).
Here  $H_{\alpha}\in \mathfrak{h}$ corresponds to $\alpha\in \mathfrak{h}^{*}$ via the isomorphism 
$\mathfrak{h} \cong \mathfrak{h}^{*}$ 
defined by
$\langle \cdot , \cdot \rangle_{\mathfrak{g} }$. Relations (\ref{unu}) and (\ref{doi}) imply that
\begin{equation}\label{Phi-expr}
\Phi (X) = - \sqrt{-1}  \sum_{j=1}^{\ell} X(\theta_{\alpha_{j}}) \tilde{H}_{\alpha_{j}}. 
\end{equation}
Using 
$\tilde{\sigma} = K \sigma + \Phi$,    $\sigma (W) \subset M\times \mathfrak{h}$, $K\vert_{M\times \mathfrak{h} }=\mathrm{Id}$ and (\ref{Phi-expr}) 
we obtain   (\ref{tilda-sigma0}).\

Relation  (\ref{tilda-sigma0}), combined with Lemma \ref{dec-iso} ii) and Lemma \ref{iso-untwisted}, 
imply 
that  for any $s\in \Gamma (M\times \mathfrak{h}_{\mathfrak{g}})$ there is $I\in \mathrm{Aut}  (E)$, with 
the $\Phi$-component $\Phi (X) = \nabla_{X}s$  ($X\in W$) and 
the $K$-component
acting as the identity on $M\times \mathfrak{h}_{\mathfrak{g}}$. Its lift $T_{I}$ maps $(W, \mathcal D , \sigma , \epsilon )$ to $(W, \mathcal D , \tilde{\sigma}, \tilde{\epsilon })$ 
where 
$\tilde{\sigma} = \sigma +\nabla\vert_{W} s$.\
Letting 
$$
s= \sum_{i=1}^{p} \theta_{i} {v}^{\prime}_{i} +\sum_{j=1}^{q} \tilde{\theta}_{j} ( w_{j} +\tau (w_{j}) ) 
+ \sqrt{-1}\sum_{j=1}^{q}  \theta^{\prime}_{j}  ( w_{j} - \tau (w_{j})), 
$$
 where   $\theta_{i} := - \mathrm{Re}\, (f_{i})$, $\tilde{\theta}_{j} := - \mathrm{Re}\, (g_{j})$ and 
 $\theta^{\prime}_{j} := \mathrm{Im}\,  (g_{j})$ we obtain 
 $$
 \tilde{\sigma} (X) = \sqrt{-1}  \sum_{i=1}^{p} X(\mathrm{Im}\, (f_{i}) ) {v}^{\prime}_{i} - \sum_{j=1}^{q} X( \bar{g}_{j}) w_{j}.
 $$
 Since $w_{j}\in \mathcal D$,  relation (\ref{sigma}) follows 
 (with $\mathrm{Im}\, f_{j}$ replaced by $f_{j}$).\

ii) From Corollary  \ref{complement_sigma:eq}, 
we can add to $\tilde{\sigma}$ any $\mathcal D$-valued $1$-form on $W$
without changing $I(\mathcal J ).$ The first  claim follows. 
 For the second claim we use Lemma~\ref{step-two} with $g_{j}  =0$ and $f_{i}$ purely imaginary.
\end{proof}

From now on we assume that $\mathcal J$ is given by Lemma \ref{cor-unu} ii). Let $L = (W, \mathcal D , \sigma , \epsilon )$ be its $(1,0)$-bundle, with $\mathcal D$ and $\sigma$ as in Lemma \ref{cor-unu} ii). 
From $W +\bar{W} = (TM)^{\mathbb{C}}$ with $W$ and $\Delta$ of constant rank,  we can  (and will)  assume that 
$M = U\times V$, 
where $U\subset  \mathbb{R}^{n-2k}$ and $V\subset  \mathbb{C}^{k}$    are open neighborhoods of the origins  and
\begin{equation}\label{W}
W:= \mathrm{span}_{\mathbb{C}} \left\{ \frac{\partial}{\partial x^{i} },  \frac{\partial}{\partial \bar{z}^{j}} \right\} ,\
\Delta := \mathrm{span}_{\mathbb{R}}\left\{ \frac{\partial}{\partial x^{i} } \right\} , 
\end{equation}
where  $(x^{i})$ are  coordinates on  $\mathbb{R}^{n-2k}$ and   $(z^{j})$ are  complex coordinates on  $\mathbb{C}^{k}\cong \mathbb{R}^{2k}$.\

When   $n=2k$  (or $\Delta =\{ 0\}$), then   $M = V$,  $W= T^{0,1}V$,   
$(TM)^{\mathbb{C}} = W\oplus \bar W$, $\mathcal D \cap \bar{\mathcal D} = \{ 0\}$
and $\sigma =0$.
Extend  $\epsilon\in \Gamma (\Lambda^{2} W^{*})$ to a real closed $2$-form on $M$.  
Applying an automorphism of $E$ of the form
 \begin{equation}\label{I-final}
 I(X + \eta + r)  = X +  ( i_{X}\beta  +\eta ) + r,
\end{equation}
for any $X\in TM$, $\eta \in T^{*}M$, $r\in \mathfrak{g}$, where $\beta \in \Omega^{2}(M)$ is closed,
we  make $\epsilon =0$
(without modifying $\mathcal D$ and
$\sigma =0$).   
The claims on $\sigma$ and $\epsilon$ from Theorem \ref{main-thm} follow. 
When $n>2k$ we need two further steps 
in order to obtain $\sigma$ and $\epsilon$ as required by Theorem \ref{main-thm}. 

\begin{lem}\label{added-1}  Assume that $n> 2k.$ The generalized complex structure 
$\mathcal J$ is equivalent to  one   with $(1,0)$-bundle  $L(W, \mathcal D , \sigma , \epsilon )$ where $W$ is given by 
(\ref{W}), $\mathcal D$  by the first relation in (\ref{w-d}), 
$\sigma =  \sum_{i=1}^{p}dx^{i}\otimes ( \sqrt{-1} v_{i}^{\prime} )$  and
\begin{equation}
\epsilon =  \sqrt{-1} \omega_{\mathrm{st}}+ \alpha, 
\end{equation}
 where  $\omega_{\mathrm{st}}$ is the standard symplectic form on $\mathbb{R}^{n-2k-p}$ (identified with the slice
$x^{i} =0$, $z^{j}=0$ for any $1\leq i\leq p$ and $1\leq j\leq k$) and the (closed) form 
${\alpha} \in \Gamma (\Lambda^{2} W^{*})$ belongs to the ideal $I$ generated by $(dx^{i},\ 1\leq i\leq p).$ 
\end{lem}

\begin{proof} We apply to $\mathcal J$  local equivalences which preserve $W$, $\Delta$  and $\mathcal D$, and map
$\sigma$ and $\epsilon $  to the required forms. We can (and will)  assume that $f_{i} (0) =0$ for any $1\leq i\leq p.$ 
Since $\{ (df_{i}) \vert_{\Delta},\ 1\leq i\leq p\}$ are linearly independent and purely imaginary, 
we can find  $f_{p+1}, \ldots , f_{n-2k} \in C^{\infty}( U\times V,  \mathbb{R})$ with $f_{i} (0) =0$ (for any $p+1\leq i\leq n-2k$) such that 
$$
\chi : = ( -\sqrt{-1} f_{1}, \ldots , -\sqrt{-1}  f_{p}, f_{p+1}, \ldots , f_{n-2k} , \mathrm{pr}^{(2)} ) : U\times V \rightarrow \mathbb{R}^{n} = \mathbb{R}^{n-2k}\times \mathbb{R}^{2k}
$$
is a coordinate system of $M$ centred at the origin, 
where $\mathrm{pr}^{(2)} : U\times V \rightarrow V$ is the projection on the second factor.
Using the Courant algebroid pullback defined by $\chi^{-1}$, 
we can assume  from the very beginning that   
$f_{i}(x, z) = x^{i}\sqrt{-1} $ for any $1\leq i\leq p$
and that $W$ and $\Delta$ have the  same form 
(\ref{W}).   In particular, $\sigma$ is as required. 
Note   that $n-p$ is even, as  
$\Delta_{0}$ is spanned by the vector fields $\frac{\partial}{\partial x^{i}}$ 
($p+1\leq i\leq n-2k$) and 
supports a non-degenerate $2$-form
(see Proposition~\ref{basis}).\  

Let $U = (-\delta  , \delta )^{n-2k}$ and $V= (-\delta, \delta )^{2k}\subset \mathbb{R}^{2k}\cong \mathbb{C}^{k} $, where $\delta >0$. 
For $(x^1,\ldots ,x^p)\in  (-\delta, \delta )^p$ and $z\in (-\delta ,\delta)^{2k}$, let 
$\epsilon_{ (x^{1}, \ldots , x^{p}, z) }$ be the restriction of $\epsilon$ to the slice 
$\{ (x^{1}, \ldots   , x^{p}) \} \times  (-\delta , \delta )^{n-2k-p} \times \{ z\}$ and remark that 
$\mathrm{Im}\, (\epsilon_{ (x^{1}, \ldots , x^{p}, z) })\in \Omega^{2} ( (-\delta , \delta )^{n-2k-p})$ is a symplectic form
(from $d\epsilon =0$ and Lemma  \ref{cor-unu} ii)).
Making $\delta$ smaller if needed,  consider a local diffeomorphism
$$
f_{(x^{1}, \ldots  , x^{p}, z)} : (-{\delta}, {\delta})^{n-2k-p} \rightarrow \mathbb{R}^{n-2k-p}
$$
which preserves the  origin, depends smoothly on  $(x^{1}, \ldots  , x^{p}, z)\in (-\delta, \delta )^{p+ 2k}$ and  satisfies 
\begin{equation}
(f_{(x^{1}, \ldots  , x^{p}, z)} )^{*} (\omega_{\mathrm{st}}) = \mathrm{Im}\, \epsilon_{(x^{1}, \ldots , x^{p}, z)} .
\end{equation}
(see Corollary~\ref{darboux-family}). Using the Courant algebroid pullback  $G^{!} : E \rightarrow E$ where 
\begin{equation}\label{domain}
G :  ( -\delta , \delta )^{p}\times ( -{\delta}, {\delta})^{n-2k-p} \times ( -\delta ,  \delta )^{2k}
\rightarrow \mathbb{R}^{n-2k}\times \mathbb{R}^{2k}
\end{equation}
is defined by 
$$
G(x^{1}, \ldots , x^{n-2k}, z) = (x^{1}, \ldots , x^{p}, f_{(x^{1}, \ldots , x^{p}, z)} (x^{p+1}, \ldots , x^{n-2k}),z),
$$
we can  (and will)  assume that 
\begin{equation}\label{im-e}
\mathrm{Im}\, \epsilon_{(x^{1}, \ldots , x^{p}, z)}  = \omega_{\mathrm{st}},
\end{equation}
while  $W$, $\Delta$, $\mathcal D$  and $f_{i} (x, z) = x^{i}\sqrt{-1}$  ($1\leq i\leq p$) are as before.\

From   (\ref{im-e}) and 
Lemma  \ref{lemma-gualt} ii) applied 
to  the restriction 
$\epsilon_{(x^{1}, \ldots  , x^{p})}$ of $\epsilon $ 
to the slice $\{ (x^{1}, \ldots , x^{p})\} \times (- {\delta}, {\delta})^{n-2k-p}\times (-\delta ,\delta )^{2k}$
we obtain a real closed $2$-form 
$$
\beta_{(x^{1}, \ldots , x^{p})}\in \Omega^{2} ( (- {\delta}, {\delta})^{n-2k-p} \times ( -\delta, \delta )^{2k})
$$
depending smoothly on $(x^{1}, \ldots , x^{p}) \in (-\delta , \delta )^{p}$ such that 
\begin{equation}\label{epsilon-beta-omega}
\epsilon\vert_{W_{0}}  =  \beta\vert_{W_{0}} + \sqrt{-1} \omega_{\mathrm{st}},
\end{equation}
where the distribution $W_0\subset W$ is defined by 
 $$
 W_{0} := \mathrm{span}_{\mathbb{C}}\left\{ \frac{\partial}{\partial x^{i}}\ (p+1\leq i\leq n-2k),\ \frac{\partial}{\partial \bar{z}^{j}}\ (1\leq j\leq k )\right\}
 $$ 
 and  $\omega_{\mathrm{st}}$ is extended  trivially to 
 $W_{0}.$ 
 We  now apply  the Courant algebroid automorphism 
with components  $( K:= \mathrm{Id}, \Phi :=0,  -2\tilde{\beta} )$
(see
(\ref{def-iso})), 
where $\tilde{\beta}\in \Omega^{2}(U\times V)$ is  real, closed and  extends  the family $\beta_{(x^{1}, \ldots  , x^{p})}.$  
From  Corollary~\ref{LI:cor} with $\gamma =0$, 
we obtain a new generalized complex structure with $(1,0)$-bundle  $\tilde{L} = L(W, \mathcal D , \sigma , \epsilon -\tilde{\beta}\vert_{W})$. 
From (\ref{epsilon-beta-omega}), $\epsilon -\tilde{\beta}\vert_{W}$ and $\sqrt{-1} \omega_{\mathrm{st}}$  (extended trivially to $W$)   
differ by a  (closed) $2$-form ${\alpha} \in \Gamma ( \Lambda^{2} W^{*})$ which belongs to the ideal $I$ generated by $( dx^{i},\ 1\leq i\leq p)$.
 \end{proof}

The next lemma concludes the proof of Theorem \ref{main-thm}. 

\begin{lem}\label{added-2} Any generalized complex structure  $\mathcal J$ as in Lemma \ref{added-1} is equivalent to one as in
Theorem \ref{main-thm}.
\end{lem}

\begin{proof} 
 Note  that $\mathcal D$ given by the first relation in (\ref{w-d}) 
describes  the $(1,0)$-space of an invariant complex structure on  $G/T$ where $\mathrm{Lie} (G) = \mathfrak{g}$ and $\mathrm{Lie} (T)
= \mathrm{span}_{\mathbb{R}}\{ v_{1},\cdots , v_{p} \}$
(see \cite{wang}), as required by Theorem \ref{main-thm}. When $n=2k$, the claim follows from 
our considerations before Lemma \ref{added-1}. Assume now that $n>2k$ and let 
$(W, \mathcal D, \sigma , \epsilon )$ as in Lemma \ref{added-1}.  We write 
$\alpha = d \left(\sum_{i=1}^{p} \mathcal H_{i} dx^{i}\right)\vert_{W}$  where $\mathcal H_{i} \in C^{\infty} (M, \mathbb{C})$, 
and,  applying a Courant algebroid automorphism of the form (\ref{I-final}),  we make 
$\mathcal H_{i}$  purely imaginary:  $\mathcal H_{i} = \Omega_{i} \sqrt{-1}$, where $\Omega_{i} \in C^{\infty} (M, \mathbb{R}).$ 
Define $s\in \Gamma (M\times \mathfrak{h}_{\mathfrak{g}}^{\mathcal D})$ by $\langle s, v_{i}^{\prime}\rangle = - \Omega_{i}$, for any $i$, and let
$I\in \mathrm{Aut} (E)$ be an automorphism with   the $\Phi$-component  $\Phi (X) =\nabla_{X}s$ ($X\in W$)  
and the $K$-component acting as the identity on $M\times \mathfrak{h}_{\mathfrak{g}}$ 
(see the proof of Lemma \ref{cor-unu} i)).
Since $c^{\Phi} =0$ (as $\Phi (X)\in \mathcal D$ for any $X$ and $\mathcal D$ is Lagrangian), 
we can choose  the $\beta$-component of $I$ to be trivial.   
From Corollary \ref{LI:cor},  the lift $L_{I, \gamma }$  with $\gamma := - \Phi\vert_{W}$ 
maps $(W, \mathcal D , \sigma , \epsilon )$  to
$(W, \mathcal D, \tilde{\sigma}, \tilde{\epsilon })$ where $\tilde{\sigma}$ and $\tilde{\epsilon}$ have the form required by
Theorem \ref{main-thm}. 
In fact, 
\begin{eqnarray*} \tilde{\sigma}&=& \sigma + \gamma +\Phi|_W = \sigma = \sum_{i=1}^{p} dx^{i} \otimes (\sqrt{-1} v_{i}^{\prime}) \\
\tilde{\epsilon} &=& \epsilon + \langle \Phi|_W \wedge \sigma \rangle = \epsilon  - \sqrt{-1} \sum_{i=1}^{p} (d \Omega_{i} )\wedge dx^{i}
=   \sqrt{-1}  \omega_{\mathrm{st}}.
 \end{eqnarray*}
We conclude the proof by renaming   $(x^{p+1}, \cdots , x^{n-2k})$ as $(y^{1}, \cdots , y^{2q})$ (where $2q:= n-2k-p$).
\end{proof}

\subsection{Dependence of parameters in Theorem \ref{main-thm}}\label{dep-param-sect}

 We now prove Corollary \ref{dep-param-cor}.  
 Let   $\mathcal J$ and $\tilde{\mathcal J}$ be two regular generalised complex structures in normal form,
with parameters  $(p,q,k,  \mathcal D , v_{i}^{\prime})$ and $(\tilde{p}, \tilde{q} ,\tilde{k},  \tilde{\mathcal D} , \tilde{v}_{i}^{\prime})$
as in Theorem \ref{main-thm}, and assume that they are related by an equivalence, which we write as 
$I \circ f^{!}$, where $I\in \mathrm{Aut} (E)$ has components   $(\Phi , K, \beta )$
and $f\in \mathrm{Diff}(M)$. We  apply  Corollary \ref{LI:cor} and  Proposition \ref{sigma0-isomorphism}.
Since  $\tilde{W} =  (df)^{-1} (W_{f}) $, $\tilde{W} \cap \overline{\tilde{W}} =  (df)^{-1}(W_{f}\cap \bar{W}_{f})$ which implies that $\tilde{k}  = k.$
Since  $\tilde{\mathcal D} = K_{x_{0}} \mathcal D$   (for any $x_{0}\in M$),  $\tilde{\mathcal D}$ belongs to the orbit $[\mathcal D ]$ of  
$\mathcal D$ under the action of 
$\mathrm{Aut} (\mathfrak{g})$.
From 
$\tilde{\mathcal D} \cap \mathfrak{g}  =  K_{x_{0}}  (\mathcal D\cap \mathfrak{g})$ 
and $p = \mathrm{dim} (\mathcal D \cap \mathfrak{g)}$ (and similarly for $\tilde{p}$)
we deduce that $\tilde{p}  = p$ are determined by $[\mathcal D ].$  
It follows that also $\tilde{q} = q$. Choosing 
in  (\ref{triv-h})  another basis of $\mathfrak{h}$ adapted to $\mathfrak{h}^{\mathcal D}$ (in particular, other vectors 
${v}_{i}^{\prime}$), we obtain another normal form for a given regular generalized complex structure. In particular, the equivalence classes of
normal forms in Theorem~\ref{main-thm}  are independent of the choice of $v_{i}^{\prime}.$ 
If $\tilde{\mathcal D} = K_{0}\mathcal D$ where $K_{0} \in \mathrm{Aut} (\mathfrak{g})$ then 
$I\in \mathrm{Aut} (E)$ with components  of the form $(K  =K_{0}, \Phi = 0, \beta =0)$, 
maps a   regular generalized complex structure in normal form  with parameters   $(p, q, k, \mathcal D , v_{i}^{\prime})$ to 
one with parameters  $(p, q, k, \tilde{\mathcal D} , K_{0}v_{i}^{\prime})$. 
This concludes 
Corollary \ref{dep-param-cor}.

Finally, we point out that, due to the conjugacy of Cartan subalgebras and the transitivity 
of the Weyl group on positive root systems for a fixed  Cartan subalgebra, the orbit space of $\mathrm{Aut}(\mathfrak g)$ on the Grassmannian 
of Lagrangian subalgebras of $\mathfrak{g}^\mathbb{C}$ reduces to the orbit space of the subgroup $\mathrm{Aut}(\mathfrak g)_{\mathfrak{h}_\mathfrak{g},R^+}$
normalising a fixed Cartan subalgebra $\mathfrak{h}_\mathfrak{g}$ of $\mathfrak g$ and preserving a fixed set of positive roots $R^+$ on 
the Grassmannian of maximally isotropic subspaces of $\mathfrak{h}=\mathfrak{h}_\mathfrak{g}^\mathbb{C}$. The latter orbit space
of $\mathrm{Aut}(\mathfrak g)_{\mathfrak{h}_\mathfrak{g},R^+}$ coincides with the orbit space of the effective action by the 
group of symmetries of the Dynkin diagram of $\mathfrak g$ on the same Grassmannian.

\begin{appendix}

\section{Appendix}
\label{appendix-section}

\subsection{Families of automorphisms}\label{cartan-section}

Let  $(\mathfrak{g}, [\cdot , \cdot ]_{\mathfrak{g}}, \langle \cdot , \cdot \rangle_{\mathfrak{g}})$ be a  compact semisimple   quadratic Lie algebra and  
$\mathfrak{h}_{\mathfrak{g}}$ a Cartan subalgebra of $\mathfrak{g}$.   Let $R$ be the set of roots of $\mathfrak{g}^{\mathbb{C}}$
relative to $\mathfrak{h} := ( \mathfrak{h}_{\mathfrak{g}})^{\mathbb{C}}.$  

\begin{lem}\label{dec-iso}  i)  Let $K\in \mathrm{Aut}(M\times \mathfrak{g})$ 
which preserves $M\times \mathfrak{h}_{\mathfrak{g}}.$ 
Then (locally) $K = K_{1} K_{2}$ where $K_{i} \in \mathrm{Aut}\, (M\times \mathfrak{g})$ preserve $M\times \mathfrak{h}_{\mathfrak{g}}$,  
$K_{1}\vert_{M\times \mathfrak{h}_{\mathfrak{g}}}$ is the identity and $K_{2}$ is constant. In particular, $K\vert_{M\times \mathfrak{h}_{\mathfrak{g}}}$ is constant.\

ii) Any $K\in \mathrm{Aut} (M\times \mathfrak{g} )$ 
such that 
$K\vert_{M\times \mathfrak{h}_{\mathfrak{g}}} $  is the identity 
satisfies
\begin{equation}
K\vert_{M\times \mathfrak{g}_{\alpha }}  = e^{\theta_{\alpha}\sqrt{-1}} \mathrm{Id},\ \forall \alpha \in R, 
\end{equation}
where $\theta_{\alpha}\in C^{\infty}(M,  \mathbb{R})$, 
$\theta_{\alpha +\beta} = \theta_{\alpha} +\theta_{\beta}$ 
for all $\alpha, \beta  , \alpha +\beta \in R$ and $\theta_{-\alpha } = -\theta_{\alpha}$  for all  $\alpha \in R.$
The  $(\theta_{\alpha})_{\alpha \in R}$ are uniquely determined by their values on simple roots and the latter can be arbitrarily chosen.
\end{lem}

\begin{proof}  i)   We claim that the induced action 
$\hat{K}_{x} : R \rightarrow R$
defined by $\hat{K}_{x} (\alpha )_:= \alpha \circ (K_{x}\vert_{\mathfrak{h}})^{-1}$,  is (locally)  independent of  $x$. 
Denoting the  cardinality of $R$ by $|R|$ we write $R= \{ \alpha_{1}, \ldots , \alpha_{|R|}\}$. Let $H_{0} \in \mathfrak{h}$ be such that $\alpha_{i} (H_{0}) \neq \alpha_{j}(H_{0})$ for all
$i\neq j .$  Fix  $x_{0}\in M$ and $\alpha \in R$. Since  $\alpha \circ (K_{x_{0} }\vert_{\mathfrak{h}})^{-1}\in R$, there is $1\leq i\leq |R|$ such that
$\alpha \circ (K_{x_{0}}\vert_{\mathfrak{h}})^{-1} (H_{0}) =\alpha_{i}(H_{0}).$ We deduce that 
$\alpha \circ (K_{x_{0}}\vert_{\mathfrak{h}})^{-1} (H_{0}) \neq \alpha_{j}(H_{0})$ for all $j\neq i$ and the same holds 
with $x_{0}$ replaced by any $x\in U$ (a  small
neighbourhood of $x_{0}$).  Since $\alpha \circ (K_{x }\vert_{\mathfrak{h}})^{-1}\in R$,  we obtain that 
$\alpha \circ (K_{x }\vert_{\mathfrak{h}})^{-1} =\alpha_{i}$, or $\hat{K}_{x}(\alpha ) =\alpha_{i}.$ We deduce that
$\hat{K}_{x} : R \rightarrow R$ is independent of $x\in U$, as claimed.
We  now define 
$K_{2}\in \mathrm{Aut} (M\times \mathfrak{g})$ to be  the constant extension of $K_{x_{0}}\in \mathrm{Aut}(\mathfrak{g})$ and $K_{1} := K  K_{2}^{-1}.$
Then $(\hat{K}_{1})_{x} =  (\hat{K}_{1} )_{x_{0}} =\mathrm{Id}$ because $(K_{1})_{x_{0}} = \mathrm{Id}.$  It follows that $K_{1}\vert_{M\times \mathfrak{h}_{\mathfrak{g}}}= \mathrm{Id}$.

ii)  Since $K\vert_{M\times \mathfrak{h}_{\mathfrak{g}}}$
is the identity,    ${K}_{x} (\mathfrak{g}_{\alpha}  )=\mathfrak{g}_{\alpha} $ for any   $\alpha \in R$.  Thus,  
\begin{equation}\label{kappa}
K\vert_{M\times \mathfrak{h}_{\mathfrak{g} }} =\mathrm{Id},\ K\vert_{M\times \mathfrak{g}_{\alpha}} = h_{\alpha} \mathrm{Id},
\end{equation}
where $h_{\alpha}\in C^{\infty}(M, \mathbb{C})$,  for all $\alpha \in R.$ 
The relation
\begin{align}
\label{form-K}& K_{x}  ( [ E_{\alpha }, E_{\beta}  ]_{\mathfrak{g} } )= [ K_{x}(E_{\alpha}), K_{x}(E_{\beta} )]_{\mathfrak{g}}
\end{align}
with 
$\alpha , \beta \in R$, together with (\ref{kappa}), 
implies that  
$h_{\alpha+\beta} = h_{\alpha } h_{\beta}$ for all $\alpha , \beta$, $\alpha +\beta \in R$.
Choosing root vectors $(E_{\alpha})_{\alpha\in R}$ such that $\tau (E_{\alpha})  = - E_{-\alpha}$,  we obtain that 
$K_{x} \circ \tau = \tau \circ K_{x}$ is equivalent to 
\begin{equation}\label{aa2}
h_{-\alpha} = \bar{h}_{\alpha},\ \forall \alpha \in R.
\end{equation}
Using  
relations (\ref{knapp-rel}) 
we obtain that 
 $\langle K_{x}r, K_{x}s\rangle_{\mathfrak{g} } = \langle r, s \rangle_{\mathfrak{g} }$ 
for all $r, s \in \mathfrak{g}$ is equivalent to
$h_{\alpha} h_{-\alpha} =1$ for all  $\alpha \in R.$ 
Combined with (\ref{aa2}), the latter relation implies  that  $| h_{\alpha} | =1$. Writing $h_{\alpha} 
= e^{\theta_{\alpha}\sqrt{-1}}$ we obtain our claim.
\end{proof}

\subsection{Families of $2$-forms}\label{form-section}

\begin{lem}\label{darboux-family}
Let $\{ \omega_{z}\}_{z\in V}\in \Omega^{2}(M)$ be a family of symplectic forms on a manifold $M$ of dimension $2n$, depending smoothly
 on  $z\in V$, where $V$ is an open neighborhood 
of the origin $0\in \mathbb{R}^{k}.$ Then, for any $p_{0}\in M$, there is an open  neighborhood $U$ of $p_{0}$, 
an open  neighborhood $V^{\prime}\subset V$ of  $0\in \mathbb{R}^{k}$ and 
diffeomorphisms  $\psi_{z} : U \rightarrow  U_{z}\subset  \mathbb{R}^{2n}$ depending smoothly on $z\in V^{\prime}$, 
such that $\psi_{z} (p_{0}) =0$ and $  \omega_{z}\vert_{U}=\psi_{z}^{*}(\omega_{\mathrm{st}}) $ for any $z\in V^{\prime}$, where
$\omega_{\mathrm{st}}$ is the standard symplectic form on $\mathbb{R}^{2n}.$ 
\end{lem}

\begin{proof} The claim follows by adapting the proof of the classical Darboux theorem in symplectic geometry (see e.g.\ \cite{duff})  to smooth families of symplectic forms.
\end{proof}

Let $M:= N\times \tilde{N}$,  where  
$N:=  ( -\delta  , \delta )^{p}\subset \mathbb{R}^{p}$, 
$\tilde{N}:=    (-\delta , \delta )^{2k}\subset \mathbb{R}^{2k}\cong (\mathbb{R}^2)^k=\mathbb{C}^{k}$ and 
$\delta >0$.
Consider the distributions on $M$ defined by
\begin{equation}
W:= \mathrm{span}_{\mathbb{C}} \left\{ \frac{\partial}{\partial x^{i} },  \frac{\partial}{\partial \bar{z}^{j}} \right\},\
\Delta := \mathrm{span}_{\mathbb{R}}\left\{ \frac{\partial}{\partial x^{i} } \right\}, 
\end{equation}
where  $(x^{i})$ are  coordinates on  $(-\delta  , \delta )^{p}\subset \mathbb{R}^{p}$ and 
$(z^{j})$ are  complex coordinates on  $ (-\delta  , \delta )^{2k}\subset \mathbb{C}^{k}$.

\begin{lem}\label{lemma-gualt}  
i) Let $\epsilon \in \Gamma (\Lambda^{2} W^{*})$ be a closed  $2$-form such that the functions
$\mathrm{Im}\,  \epsilon (  \frac{\partial}{\partial x^{i} },
 \frac{\partial}{\partial x^{j} })\in C^{\infty}(M, \mathbb{R})$ 
are  independent of $z$, for any $i, j.$  
Then, locally,  there is a  real closed $2$-form $\beta \in \Omega^{2}(M)$ such that 
\begin{equation}
\epsilon = ( \beta +\sqrt{-1}  \mathrm{Im}\,  (\epsilon \vert_{\Delta}))\vert_{W}.
 \end{equation}
Here $\beta$ is extended by complex linearity to $(TM)^{\mathbb{C}}$    and 
$ \mathrm{Im}\,  (\epsilon \vert_{\Delta})\in \Gamma (\Lambda^{2} \Delta^{*} )$  is 
extended by complex linearity to $\Delta^{\mathbb{C}}$ and then trivially to
$(TM)^{\mathbb{C}}$.\

ii) The above statement holds with parameters. More precisely, if $\epsilon = \epsilon_{(x^{p+1}, \ldots  , x^{n})}$ as above depends smoothly on some parameters 
$(x^{p+1}, \ldots , x^{n})$
then (making $\delta$ smaller if necessary), we can choose $\beta = \beta_{(x^{p+1}, \ldots , x^{n})}$ to depend smoothly on 
$(x^{p+1}, \ldots , x^{n})$  as well. 
\end{lem}

 \begin{proof} 
 The proof consists in checking that the arguments of~\cite[Section 4.7]{gualtieri-thesis} can be done with smooth dependence on 
 parameters.
  \end{proof}
 
 \subsection{Isotropic subspaces}\label{basis-section}

Let $(V, g ) $ be a real vector space of even dimension $2n$ with  a non-degenerate symmetric bilinear form.
Let  $\tau : V^{\mathbb{C}}
\rightarrow V^{\mathbb{C}}$ be the conjugation of $V^{\mathbb{C}}$ determined by the real form $V$.

\begin{prop}\label{max-iso-chevalley}
Let $\mathcal D$ be a maximal isotropic subspace  of $(V^{\mathbb{C}}, g)$. There is a basis 
(called {\cmssl adapted to $\mathcal D$}) 
$$
B:= \{ v_{1}, \ldots , v_{p},  w_{1}, \ldots , w_{q}, 
{v}^{\prime}_{1}, \ldots , {v}^{\prime}_{p}, \tau (w_{1}), \ldots , \tau (w_{q}) \}
$$
of $V^{\mathbb{C}}$  ($p+q =n$), such that  $v_{a}, {v}^{\prime}_{b} \in V$ for any $a, b$, 
$$
\mathcal D =\mathrm{span}_{\mathbb{C}}\{ v_{1}, \ldots , v_{p}, w_{1}, \ldots , w_{q}\}
$$
and the only non-zero scalar products of vectors from  $B$ are 
\begin{equation}
g (v_{a}, {v}^{\prime}_{a} ) = 1,\ g (w_{c}, \tau (w_{c}) ) = \epsilon_{c} \in \{ \pm 1\}
\end{equation}
where $1\leq a\leq p$, $ 1\leq c\leq q$.
\end{prop}

\begin{proof}
Let  $\{ v_{1}, \cdots , v_{p} \} $ be a  basis of $N:= \mathcal D \cap V$ and $\{ {v}^{\prime}_{1}, \cdots  , {v}^{\prime}_{p}\}$ 
a system of linearly independent vectors 
which generate an isotropic subspace $P\subset V$, such that
$g( v_{a}, {v}^{\prime}_{b}) =\delta_{ab}$,
for any $1\leq a, b\leq p$ (see  \cite[I.3.2, page 77]{chevalley}). 
Then 
$V$ decomposes into  a direct sum of $N\oplus P$ and its orthogonal complement
$( N\oplus P)^{\perp}$ and $\mathcal D$ decomposes  into 
$N^{\mathbb{C}} \oplus \tilde{\mathcal D}$ and 
$\tilde{\mathcal D} := \mathcal D \cap ( ( N\oplus P)^{\perp})^{\mathbb{C}}$.
Moreover, 
$\tilde{\mathcal D}$  is a maximal isotropic  subspace of    the even dimensional vector space  $( (N\oplus P)^{\perp} )^{\mathbb{C} }$ and 
$\tilde{\mathcal D} \cap \tau ( \tilde{\mathcal D})  = \{ 0\}$. 
The vectors  $\{ w_{1}, \ldots , w_{q} \}$ form a basis of 
$\tilde{\mathcal D}$ and can be constructed inductively using linear algebra. 
\end{proof}

\end{appendix}

V. Cort\'es: vicente.cortes@math.uni-hamburg.de\

Department of Mathematics and Center for Mathematical Physics, University of Hamburg,  Bundesstrasse 55, D-20146, Hamburg, Germany.\\

L. David: liana.david@imar.ro\

Institute of Mathematics  `Simion Stoilow' of the Romanian Academy,   Calea Grivitei no.\ 21,  Sector 1, 010702, Bucharest, Romania.

\end{document}